\documentclass[a4paper,reqno]{amsart}  

\usepackage[T1]{fontenc}         
\usepackage[applemac]{inputenc} 
\usepackage{amsthm}
\usepackage{amsmath}
\usepackage{amsfonts}
\usepackage{amssymb}
\usepackage[arrow, matrix, curve]{xy}
\usepackage{picinpar, graphicx}
\usepackage{geometry}
\usepackage{epic}
\usepackage{color}

\newtheorem{Thm}{Theorem}

\newtheorem{Lem}{Lemma}
\newtheorem{Prop}{Proposition}

\begin{document}

\renewcommand{\thefootnote}{}
\footnotetext{Research partially supported by a CRM-ISM post-doctoral Fellowship, by NSERC Grant RGPIN 105490-2011 and by the Minist\'erio de Ci\^encia e Tecnologia, Brazil, CNPq Proc. No. 303774/2009-6.}

\title[Second-order equations and immersions of pseudo-spherical surfaces]{Second-order equations and local isometric immersions of pseudo-spherical surfaces}
\author{Nabil Kahouadji, Niky Kamran  and Keti Tenenblat}
\date{}

\begin{abstract}
We consider the class of differential equations that describe pseudo-spherical surfaces of the form  
$u_t=F(u,u_x,u_{xx})$ and $u_{xt}=F(u, u_x)$ given in Chern-Tenenblat \cite{ChernTenenblat} and Rabelo-Tenenblat \cite{RabeloTenenblat90}.  We answer the following question: 
Given a pseudo-spherical surface determined by a solution $u$ of such an equation, do the coefficients of the second fundamental form of the local isometric immersion in $\mathbb{R}^3$ depend on a jet of finite 
order of $u$? We show that, except for the sine-Gordon equation, where the coefficients  
depend on a jet of order zero, for all other differential equations, whenever such an immersion exists, the 
coefficients are universal functions of $x$ and $t$, independent  of $u$.    
\end{abstract}

\maketitle

\section{Introduction}

The class of partial differential equations describing pseudo-spherical surfaces, which has been defined and studied in depth in a foundational paper by Chern and Tenenblat \cite{ChernTenenblat}, contains a large subclass of equations enjoying remarkable integrability properties, such as the existence of infinite hierarchies of conservation laws, B\"acklund transformations and associated linear problems. Recall that a partial differential equation 
\begin{equation}\label{pde}
\Delta(t,x,u,\frac{\partial u}{\partial x},\frac{\partial u}{\partial t},\ldots,\frac{\partial^{k}u}{\partial t^{l}\partial x^{k-l}})=0,
\end{equation}
is said to describe pseudo-spherical surfaces if there exist $1$-forms 
\begin{equation}\label{forms}
\omega^{i}=f_{i1}dx+f_{i2}dt,\quad 1\leq i \leq3,
\end{equation}
where the coefficients $f_{ij},\,1\leq i \leq 3,\,1\leq j\leq 2,$ are smooth functions of $t,x,u$ and finitely many derivatives of $u$ with respect to $t$ and $x$, such that the structure equations 
\begin{equation}\label{struct}
d\omega^{1}=\omega^{3} \wedge\omega^{2},\,d\omega^{2}=\omega^{1} \wedge\omega^{3},\,d\omega^{3}=\omega^{1} \wedge\omega^{2}
\end{equation}
hold if, and only if, $u$ is a solution of (\ref{pde}) for which $\omega^{1}\wedge \omega^{2}\neq 0$. In other words, every smooth solution of an equation (\ref{pde}) describing pseudo-spherical surfaces defines on its domain $U\subset \mathbb{R}^{2}$ a Riemannian metric 
\begin{equation}\label{metric}
ds^{2}=(\omega^{1})^{2}+(\omega^{2})^{2},
\end{equation}
of constant Gaussian curvature equal to $-1$, with $\omega^{3}$ being the Levi-Civita connection $1$-form of the metric (\ref{metric}). 

One of the most important examples of a partial differential equation describing pseudo-spherical surfaces is the sine-Gordon equation
\begin{equation}
\frac{\partial^{2}u}{\partial t\partial x}=\sin u,
\end{equation}
for which a choice of $1$-forms (\ref{forms}) satisfying the structure equations (\ref{struct}) is given by
\begin{eqnarray}\label{omega1}
\omega^1 &=& \cos \frac{u}{2}( dx+dt),\\\label{omega2}
\omega^2 &=& \sin \frac{u}{2} (dx - dt),\\\label{omega3}
\omega^3 &=& \frac{u_{x}}{2} dx - \frac{u_{t}}{2} dt.
 \end{eqnarray}
It should be noted that this choice of $1$-forms is by no means unique. In particular, we could also have used 
\begin{eqnarray}
\omega^1 &=& \frac{1}{\eta}\sin u \,dt,\\
\omega^2 &=& \eta\, dx+\frac{1}{\eta}\cos u \,dt,\\
\omega^3 &=& u_{x}\,dx,
 \end{eqnarray}
where $\eta$ is a continuous non-vanishing real parameter. This continuous parameter is closely related to the parameter appearing in the classical B\"acklund transformation for the sine-Gordon equation and accounts for the existence of infinitely many conservation laws for the sine-Gordon equation. More generally, partial differential equations (\ref{pde}) which describe pseudo-spherical surfaces and for which one of the components $f_{ij}$ (say $f_{21}$) can be chosen to be a continuous parameter will be said to describe $\eta$ pseudo-spherical surfaces.


In \cite{ChernTenenblat}, Chern and Tenenblat  
provided a complete classification of the evolution equations of the form 
\begin{equation}\label{evoleq}
u_t= F(u,u_x,..., \partial u/\partial x^k), 
\end{equation} 
which describe pseudo-spherical surfaces under the assumption that $f_{21}=\eta$, where $\eta$ is a real parameter, providing an extensive class of non-linear partial differential equations, in two independent 
variables, describing pseudo-spherical surfaces. Rabelo in \cite{Rabelo88}, \cite{Rabelo89} characterized equations of the form $u_{xt}=F(u, u_x,..., \partial u/\partial x^k)$, with $f_{21}=\eta$. The complete classification for equations of type $u_{xt}=F(u,u_x)$ 
and  $u_t=u_{xxx}+G(u,u_x,u_{xx})$ was given in \cite{RabeloTenenblat90} and \cite{RabeloTenenblat92}, respectively. 

In general, the importance of the class of differential equations that describe pseudo-spherical surfaces  is due to the fact that such a differential equation is always the integrability condition of a linear system of differential equations, which may be used in the inverse scattering method to solve the differential equation 
(see for example \cite{BealsRabeloTenenblat}, where the method was applied to a subclass of equations obtained in \cite{Rabelo89}).
While the assumption of $f_{21}=\eta$ is natural in the context of the inverse scattering method,  
the problem of classifying the differential equations describing pseudo-spherical surfaces, without any other assumption, is important in its own right and was considered by Kamran and Tenenblat in \cite{KamranTenenblat}, 
where one can find a complete classification of evolution equations of the form (\ref{evoleq}) 
which describe pseudo-spherical surfaces, as opposed to $\eta$ pseudo-spherical surfaces.   
These results provide a systematic way of verifying if a given differential equation of this type describes pseudo-spherical surfaces. The results obtained in \cite{KamranTenenblat} were extended by Reyes in 
\cite{Reyes98} to differential equations of the form 
$u_t= F(x,t, u,u_x,..., \partial u/\partial x^k)$.  
The concept of a differential equation that describes pseudo-spherical surfaces was extended by Ding and Tenenblat in \cite{DingTenenblat} to a system of differential equations that describes constant curvature surfaces (pseudo-spherical and also spherical), where classification results for such systems were obtained. 
More recently, in order to determine new classes of differential equations that describe pseudo-spherical surfaces, as a consequence of \cite{KamranTenenblat}, assuming that 
$f_{21}$ and $f_{31}$ are linear combinations of $f_{11}$,  Gomes \cite{Gomes} classified and obtained large new classes of such equations by considering fifth order equations of type (\ref{evoleq}).

We should point out that the classification results mentioned above, contain not only general statements, but also examples of interesting new and well kown non linear differential equations.  
Other aspects of the theory of differential equations which describe pseudo-spherical surfaces and its applications thereof can be found in 
\cite{CavalcanteTenenblat}, \cite{GorkaReyes}, \cite{JorgeTenenblat}, \cite{FoursovOlverReyes}, \cite{Reyes98}-\cite{Reyes106}.

A classical theorem in the theory of surfaces states that any pseudo-spherical surface can be locally isometrically immersed into three-dimensional Euclidean space $\mathbb{E}^{3}$. This result can thus be applied to the metrics arising from the solutions $u$ of any partial differential equation (\ref{pde}) describing pseudo-spherical surfaces,  thereby associating to any solution $u$ a local isometric immersion of a metric with constant Gaussian curvature equal to $-1$. This theorem is however largely an existence result, which does not give an explicit expression for the second fundamental form of the local isometric immersion. It is therefore a most remarkable property of the sine-Gordon equation that the second fundamental form of any such immersion can be expressed in closed form as a function of $u$ and finitely many derivatives. Indeed, let us first recall that the components $a,b,c$ of the second fundamental form of any local isometric immersion of a metric of constant curvature equal to $-1$ into $\mathbb{E}^{3}$ are defined by the $1$-forms $\omega^{3}_{1},\omega^{3}_{2}$ according to 
\begin{equation}\label{w13_w23}
\omega^3_{1} = a\omega^1+b\omega^2, \quad \omega^3_{2} = b\omega^1+c\omega^2,
\end{equation}
where these forms satisfy the structure equations
\begin{equation}\label{Codazzi}
d\omega^3_{1} = -\omega^3_{2}\wedge\omega^3, \quad d\omega^3_{2} = \omega^3_{1}\wedge\omega^3,
\end{equation}
and the Gauss equation 
\begin{equation*}
ac-b^2=-1.
\end{equation*}
 For the sine-Gordon equation, with the choice of $1$-forms $\omega^{1},\,\omega^{2}$ and $\omega^{3}$ given by (\ref{omega1}), (\ref{omega2}) and (\ref{omega3}),  it is easily verified that the $1$-forms $\omega^{3}_{1},\omega^{3}_{2}$ are given by
\begin{eqnarray*}
\omega^3_{1} &=& \sin\frac{u}{2} (dx+dt) =  \tan \frac{u}{2}\omega^1,\\
\omega^3_{2} &=& -\cos\dfrac{u}{2} (dx - dt) = -\cot \frac{u}{2}\omega^2.
\end{eqnarray*}
In general, given a partial differential equation (\ref{pde}) describing pseudo-spherical surfaces, it is straightforward to derive a set of necessary and sufficient conditions, in terms of the coefficients $f_{ij}$ of the $1$-forms (\ref{forms}), for $a,b$ and $c$ to be the components of the second fundamental form of a local isometric immersion corresponding to a solution of (\ref{pde}). We write  
\begin{eqnarray}
d\omega^3_{1}= \Big( db(e_{1}) -  da(e_{2})\Big)\omega^1\wedge\omega^2 - a\,\omega^2\wedge\omega^3 +b\,\omega^1\wedge\omega^3, \label{dw13}\\ 
d\omega^3_{2}= \Big( dc(e_{1}) -  db(e_{2})\Big)\omega^1\wedge\omega^2 - b\,\omega^2\wedge\omega^3 +c\,\omega^1\wedge\omega^3, \label{dw23}
\end{eqnarray}
where $(e_{1}, e_{2})$ is the pair of vector fields dual to the coframe $(\omega^1, \omega^2)$, given by 
\begin{equation*}
 \left|\begin{array}{cc}f_{11} & f_{21} \\f_{12} & f_{22}\end{array}\right|e_{1} = f_{22}\partial_{x} - f_{21}\partial_{t}, \quad  \left|\begin{array}{cc}f_{11} & f_{21} \\f_{12} & f_{22}\end{array}\right|e_{2} = -f_{12}\partial_{x} + f_{11}\partial_{t}.
 \end{equation*}
Thus, using the notation $D_{t}$ and $D_{x}$ for the total derivative operators, we obtain
\begin{eqnarray}\label{Eq1}
 f_{11} D_{t}a + f_{21}D_{t}b - f_{12}D_{x}a - f_{22}D_{x}b - 2b \left|\begin{array}{cc}f_{11} & f_{31} \\f_{12} & f_{32}\end{array}\right|+ (a-c)\left|\begin{array}{cc}f_{21} & f_{31}\\f_{22} & f_{32}\end{array}\right|   = 0,\\\label{Eq2}
 f_{11} D_{t}b + f_{21}D_{t}c - f_{12}D_{x}b - f_{22}D_{x}c +(a-c) \left|\begin{array}{cc}f_{11} & f_{31} \\f_{12} & f_{32}\end{array}\right|+ 2b\left|\begin{array}{cc}f_{21} & f_{31}\\f_{22} & f_{32}\end{array}\right|  = 0,
 \end{eqnarray}
where $a, b$ and $c$, which are assumed to depend on $t,x,u$ and finitely many derivatives of $u$ with respect to $t$ and $x$, satisfy the Gauss equation 
\begin{equation}\label{Gauss}
ac-b^2=-1.
\end{equation}

In view of the above discussion, it is is natural to ask the following question: {\em Do there exist equations other than the sine-Gordon equation within the class of partial differential equations describing pseudo-spherical (or $\eta$ pseudo-spherical) surfaces, for which the components $a,b,c$ of the second fundamental form of the local isometric immersion depend on a jet of finite order of $u$, that is on $x,t, u$ and finitely many derivatives of $u$?} 

If such equations were to exist, they would have an important geometric property in common with the sine-Gordon equation. In this paper, we give a complete answer to the above question in the case of second-order hyperbolic equations of the form
\begin{equation}\label{hyperbolic}
u_{xt}  = F(u,u_{x}),
\end{equation}
and evolution equations of the form
\begin{equation}\label{evolution}
u_{t}=F(u,u_{x},u_{xx}),
\end{equation}
which describe $\eta$ pseudo-spherical surfaces as in \cite{RabeloTenenblat90} and \cite{ChernTenenblat} .

We begin with the case of evolution equations (\ref{evolution}), for which our main result is the following:
\begin{Thm}\label{EvolRes}Except for second-order evolution equations of the form 
\begin{equation}\label{EqException}
u_{t} = \dfrac{f_{12,u_{x}}}{f_{11,u}}u_{xx} + \dfrac{f_{12, u}}{f_{11, u}}u_{x} \mp \dfrac{\lambda f_{11} - \eta f_{12}}{f_{11, u}}, 
\end{equation} where $f_{11, u} \neq 0$ and $f_{12, u_{x}} \neq 0$, there exists no second-order evolution equation describing $\eta$ pseudo-spherical surfaces, given as in \cite{ChernTenenblat}, with the property that the coefficients of the second fundamental forms of the local isometric immersions of the surfaces associated to the solutions $u$ of the equation depend on a jet of finite order of $u$.  Moreover, the coefficients of the second fundamental forms of the local isometric immersions of the surfaces determined by the solutions $u$ of (\ref{EqException}) are universal, i.e., they are universal functions of $x$ and $t$, independent of $u$. 
\end{Thm}

Theorem \ref{EvolRes} suggests  that there is no real analogue of the sine-Gordon equation within the class of second-order evolution equations describing $\eta$ pseudo-spherical surfaces, from the perspective of the local isometric immersions of pseudo-spherical surfaces associated to their solutions. Indeed, even though the special class of evolution equations (\ref{EqException}) has the property that the components of the second fundamental forms of the immersions associated to its solutions depend on jets of finite order of $u$, this dependence is given in terms of functions of $x$ and $t$  for all choices of solutions $u$ (see Proposition \ref{propabcexpluniversal}).

 The results for second-order hyperbolic equations (\ref{hyperbolic}) are similar, with the notable exception that they single out the sine-Gordon equation as the only equation, up to constants, for which the second fundamental form of the local isometric immersion is not universal. In order to state these results, we begin by recalling the classification theorem proved by Rabelo and Tenenblat \cite{RabeloTenenblat90} for equations (\ref{hyperbolic}) describing pseudo-spherical surfaces:


\begin{Thm}[Rabelo \& Tenenblat \cite{RabeloTenenblat90}] \label{RTClassification} Let $F$ be a differentiable function defined on an open \linebreak connected subset $U \subset \mathbb{R}^2$. An equation
\begin{equation*}
u_{xt}  = F(u,u_x)
\end{equation*}
describes an $\eta$ pseudo-spherical surface for $\eta \in \mathcal{P}\subset \mathbb{R}$, where $\mathcal{P}$ is a dense subset of $\mathbb{R}$ and  $F$ independent of $\eta$ if, and only if, $F$ satisfies one of the following:
\begin{enumerate}
\item[i)]$F$ is independent of $u_x$ and $F''(u) + \alpha F(u)=0$, $U=\mathbb{R}^2$, $\mathcal{P} = \mathbb{R}\setminus\{0\}$, and $\alpha$ is a non-zero real constant. 

\item[ii)] $F=\nu e^{\delta u}\sqrt{\beta + \gamma u_x^2}$, where $U=\{(u,z) \in \mathbb{R}^2; \beta + \gamma z^2 >0\}$, $\mathcal{P}= \mathbb{R}, \delta, \gamma, \beta, \nu $ are real constants, with $\delta, \gamma, \nu$ nonzero, and $\beta = 0$ when $\gamma=1$; or 

\item[iii)] $F =  \lambda  u + \zeta u_x + \tau$, where $U= \mathbb{R}^2$, $\mathcal{P} = \mathbb{R}\setminus \{0\}$, and $\lambda, \zeta, \tau$ are real constants. 
\end{enumerate}
\end{Thm}
The expressions of functions $f_{ij}$ of the 1-forms $\omega^i$ for each equation of Theorem \ref{RTClassification} 
are recalled in Section 4 (Lemmas \ref{LemF}-\ref{LemLin}). We are now ready to state our main result for the case of second-order hyperbolic equations (\ref{hyperbolic}).

\begin{Thm}\label{HyperbRes} Let $F$ be an equation of the form $u_{xt} = F(u,u_x)$ that describes  $\eta$ pseudo-spherical surfaces as in  \cite{RabeloTenenblat90}.  
\begin{enumerate}
\item If $F$ is independent of $u_x$ and satisfies $F''(u) + \alpha F(u)=0$, where $\alpha$ is a positive real constant, then there exists a local isometric immersion  in $\mathbb{R}^3$ of the pseudo-spherical surface determined by a solution $u$, for which the coefficients  of the second fundamental form depend on a  jet of finite  order of $u$ if, and only if, they depend on the jet of order zero.   
\item If  $F= \lambda  u + \zeta u_x + \tau $,  then there exists a local isometric immersion  in $\mathbb{R}^3$ of the pseudo-spherical surface determined by a solution $u$, for which the coefficients  of the second fundamental form depend on a  jet of finite  order of $u$ if, and only if,
 $\lambda$, $\xi$ and $\tau$ do not vanish simultaneously,   
and the coefficients are independent of $u$, that is they are universal functions of $x$ and $t$.
\item For the remaining equations, that is, if $F$ is independent of $u_x$ and satisfies $F''(u) + \alpha F(u)=0$, where $\alpha$ is a negative real constant, $F =\nu e^{\delta u}\sqrt{\beta + \gamma u_x^2} $ and $F=0$,  there is no local isometric immersion of the pseudo-spherical surface determined by a solution $u$,  for which the coefficients of the second fundamental form depend on a jet of finite order of $u$. 
\end{enumerate}
\end{Thm}

The coefficients of the second fundamental form of the local isometric immersions stated in Theorem  \ref{HyperbRes} are given explicitly in Section 4 (Propositions  \ref{Propi} and \ref{Propiii}). 
Theorem \ref{HyperbRes} shows likewise that when viewed through the perspective of the local isometric immersions associated to its solutions, the sine-Gordon equation occupies a special position within the class of hyperbolic equations (\ref{hyperbolic}) as the unique equation, up to normalization constants, for which the coefficients of the second fundamental form of the local isometric immersion of the surface determined by a solution $u$, depends on a jet of finite order of $u$, without being universal, {\it i.e.} independent of $u$. 


While Theorems \ref{EvolRes} and \ref{HyperbRes} give a complete answer to the general question we have raised in this paper in the case of second-order evolution equations (\ref{evolution}) and second-order hyperbolic equations (\ref{hyperbolic}), the question still remains open for all the other classes of equations describing pseudo-spherical surfaces.  We believe that it should be possible to extend the proof of Theorem \ref{EvolRes} to the case of $k$-th order evolution equations with $k\geq 3$ in order to obtain a similar result to the effect that all the second-fundamental forms that depend only on jets of finite order of the solutions of evolution equation should be universal.

Our paper is organized as follows. In Section \ref{ClassifEvol}, we recall the results of Chern and Tenenblat \cite{ChernTenenblat} on the classification of evolution equations describing pseudo-spherical surfaces and use these to give an analogue of the normal forms of Theorem \ref{RTClassification} for the case of second-order evolution equations (\ref{evolution}). These normal forms are then used as the starting point in Section \ref{ProofEvol} of the proof of Theorem \ref{EvolRes}. Section \ref{HyperbProof} is devoted to the proof of Theorem \ref{HyperbRes}. The proofs involve a careful analysis of the possible dependence on higher-order jets of $u$ of the solutions of the system of differential constraints (\ref{Eq1}) and  (\ref{Eq2}) that must be satisfied by components $a,b,c$ of the second fundamental form, together with the algebraic constraint given by the Gauss equation (\ref{Gauss}).

\section{The classification of second-order evolution equations describing $\eta$ pseudo-spherical surfaces}\label{ClassifEvol}
In \cite{ChernTenenblat}, Chern and Tenenblat obtained necessary and sufficient conditions in the form of differential equations on the functions $f_{ij}$ for the existence of an evolution equation of the form 
\begin{equation}\label{UtOrderk}
\frac{\partial u}{\partial t}=F(u,\ldots,\frac{\partial ^{k}u}{\partial x^{k}}),
\end{equation}
which describes $\eta$ pseudo-spherical surfaces, i.e., with $f_{21} = \eta$, where $\eta$ is a nonzero parameter. They also performed a complete classification of the evolution equations of the form (\ref{UtOrderk}) which describe  $\eta$ pseudo-spherical surfaces.  
They obtained four classes of evolution equations (Theorems 2.2 to 2.5 in \cite{ChernTenenblat}). These four classes of equations are determined algebraically  by $f_{11}, f_{31}, f_{22}$ and their derivatives, up to some differential constraints.  In what follows, we consider only second-order evolution equations of the form (\ref{UtOrderk}) and solve the differential constraints that $f_{11}, f_{31}$ and $f_{22}$ must satisfy in order for (\ref{UtOrderk}) to describe $\eta$ pseudo-spherical surfaces. We shall deal with two of the  four classes (Theorems 2.2 and 2.4 in \cite{ChernTenenblat}) since the two remaining classes of evolution equations (Theorems 2.3 and 2.5 in \cite{ChernTenenblat}) lead to evolution equations of the first order, when $k=2$. 

It will be convenient to introduce the following notation for the spatial derivatives of $u$ (used in \cite{ChernTenenblat} and also in \cite{KamranTenenblat}),
\begin{equation*}
z_{i} = \dfrac{\partial^{i}u}{\partial x^{i}}, \quad 0 \leqslant i \leqslant k,
\end{equation*}
and to view $(x, t, z_{0}, z_{1}, \dots, z_{k})$ as local coordinates in an open subset $U$ of a manifold. 

\begin{Lem}[Chern \& Tenenblat \cite{ChernTenenblat}] Consider a second-order evolution equation of the form $z_{0,t}= F(z_{0},z_{1},z_{2})$ which describes an $\eta$ pseudo-spherical surface with associated forms $\omega^{i} = f_{i1}dx + f_{i2}dt$. If $f_{ij}$ are differentiable functions of $z_{0}, z_{1}, z_{2}$, then
\begin{eqnarray}\label{NC1}
f_{ij,z_2}=0, &  &f_{11,z_{1}} =f_{31,z_{1}} =f_{22,z_{1}} =0, 
\\\label{NC2}
& & f_{11,z_{0}}^2 + f_{31,z_{0}}^2 \neq 0. 
\end{eqnarray}
\end{Lem}

In order to state the results, we introduce the following notation
\begin{eqnarray}\label{LHPM}
\begin{array}{llllll}H &=& f_{11}f_{11,z_{0}} - f_{31}f_{31,z_{0}},& L &=& f_{11}f_{31,z_{0}} - f_{31}f_{11,z_{0}}, \\P&=&f_{11,z_{0}}f_{31,z_{0}z_{0}} -  f_{31,z_{0}}f_{11,z_{0}z_{0}},  & M&=&f_{31,z_{0}}^2 - f_{11,z_{0}}^2.\end{array}
\end{eqnarray}

\begin{Lem}\label{LemHLneq0}Let $f_{ij}$, $1\leqslant i \leqslant 3$, $1\leqslant j \leqslant 2$, be differentiable functions of $z_{0}, z_{1}, z_{2}$ such that (\ref{NC1}) and (\ref{NC2}) hold and $f_{21}=\eta$ a nonzero parameter.  Suppose $HL\neq 0$. Then $z_{0,t} = F(z_{0},z_{1},z_{2})$ describes an $\eta$ pseudo-spherical surface with associated $1$-forms $\omega^i = f_{i1}dx + f_{i2}dt$, if and only if
\begin{equation}\label{EET1}
F = \dfrac{\mp f_{22,z_{0}}}{\eta\sqrt{1-\alpha^2}f_{11,z_{0}}}z_{2}\mp \dfrac{f_{22,z_{0}z_{0}}}{\eta\sqrt{1-\alpha^2}f_{11,z_{0}}}z_{1}^2 +\Big( \dfrac{(\eta^2 + f_{11}^2 - f_{31}^2)f_{22,z_{0}}}{\eta \Big[(1-\alpha^2)f_{11} \mp \alpha \eta \sqrt{1-\alpha^2}\Big]f_{11,z_{0}}} + \dfrac{f_{22}}{\eta}\Big)z_{1},
\end{equation}
and
\begin{eqnarray}\label{f31-1}
f_{31} &=& \alpha f_{11} \pm \eta \sqrt{1-\alpha^2},\\\label{f12-1}
f_{12} &=& \dfrac{f_{11}f_{22}}{\eta} \mp \dfrac{f_{22,z_{0}}}{\eta\sqrt{1-\alpha^2}}z_{1}, \\\label{f32-1}
f_{32} &=&  \dfrac{(\alpha f_{11} \pm \eta\sqrt{1-\alpha^2})f_{22}}{\eta} \mp \dfrac{\alpha f_{22,z_{0}}}{\eta\sqrt{1-\alpha^2}}z_{1}, 
\end{eqnarray}
where $f_{22,z_{0}}\neq 0$, $f_{11,z_{0}}\neq 0$, and $\alpha^2 < 1$.
\end{Lem}

\begin{proof}If $k=2$, Theorem 2.2 in \cite{ChernTenenblat} gives the general expression of second-order evolution equations $z_{0,t}=F$ which describe $\eta$ pseudo-spherical surfaces, namely    
\begin{equation}
   \label{2.10CT}
F=\dfrac{1}{L} \sum_{i=0}^1 z_{i+1}B_{z_i} +\dfrac{1}{HL}\left(-z_1\dfrac{L}{\eta}+f_{31}^2-f_{11}^2\right)z_1 A^0  
+\dfrac{B}{HL}(z_1M+\eta L)+z_1\dfrac{f_{22}}{\eta},  
\end{equation}
where  
\[
B=f_{22,z_0}z_1, \quad A^0=\dfrac{1}{L}(-z_1 P+\eta M)B_{z_1}+f_{22,z_1}H,
\]
where the functions $f_{12}$ and $f_{32}$ are given by  
\begin{eqnarray}
f_{12}=\dfrac{f_{11}f_{22}}{\eta}+\frac{1}{H}\left(-\dfrac{f_{11}A^0}{\eta}z_1+f_{31,z_0}B\right), \nonumber \\
   \label{2.11CT}\\ 
f_{32}=\dfrac{f_{31}f_{22}}{\eta}+\frac{1}{H}\left(-\dfrac{f_{31}A^0}{\eta}z_1+f_{11,z_0}B\right), \nonumber
\end{eqnarray}
 and where (2.12) in \cite{ChernTenenblat} gives  
 two differential equations  that the functions $f_{11}, f_{31}$ and $f_{22}$ must satisfy. When $k=2$, these equations  reduce to  
\begin{eqnarray}
& \dfrac{L}{\eta}f_{22,z_0}-\dfrac{L}{\eta}\left(z_1\dfrac{A^0}{H}\right)_{z_0} +A^0 +\dfrac{M}{H}B_{z_0}+\dfrac{B}{H^2}(LP+M^2)=0, 
   \label{2.12CTj=0} \\
& \dfrac{L}{\eta}f_{22,z_1}-\dfrac{L}{\eta}\left(z_1\dfrac{A^0}{H}\right)_{z_1} +\dfrac{M}{H}B_{z_1}=0.\label{2.12CTj=1}
\end{eqnarray}
 
 If $L\neq 0$, then the differential equation (\ref{2.12CTj=1}) leads to $Pf_{22,z_{0}}=0$. The vanishing of $f_{22,z_{0}}$ contradicts the fact that $F$ is a second order evolution equation. We conclude then that
\begin{eqnarray}
f_{22,z_{0}}\neq 0, \quad P=0.
\end{eqnarray}
Differentiating (\ref{2.12CTj=0})  with respect to $z_{1}$ leads to $-L(M/HL)_{z_{0}} +M^2/L^2 = 0$
and hence, the differential equation (\ref{2.12CTj=0})  leads to
\begin{equation}\label{M}
M=-\dfrac{L^2}{\eta^2}.
\end{equation}
The vanishing of $P$ implies that 
\begin{equation}\label{f31alphabeta}
f_{31} = \alpha f_{11} + \beta, \quad   \text{ where }  \alpha, \beta \in \mathbb{R}.
\end{equation} We have then
\begin{eqnarray}\label{HL}
L= -\beta f_{11,z_{0}}, \quad H= [(1-\alpha^2)f_{11} - \alpha \beta]f_{11,z_{0}}.
\end{eqnarray}
The non-vanishing of $L$ implies that $\beta \neq 0$ and $f_{11,z_{0}} \neq 0$. Substituting (\ref{f31alphabeta}) in  (\ref{M}) and  in the expression of $M$ as in (\ref{LHPM}) leads to 
\begin{equation}\label{betaetaalpha}
\beta^2 = \eta^2(1-\alpha^2).
\end{equation}
The non-vanishing of $\beta$ and $\eta$, and the latter equation imply that $\alpha \in (-1, 1)$. Finally, substituting $\beta = \pm \eta \sqrt{1-\alpha^2}$, $P=0$ and (\ref{f31alphabeta}) in the expressions (\ref{2.10CT}) and (\ref{2.11CT})  
leads to expressions (\ref{EET1}), (\ref{f31-1}), (\ref{f12-1}), and (\ref{f32-1}).\end{proof}

If $HL = 0$, then there are three classes of evolution equations to consider, which are given  in Theorems 2.3-2.5 in \cite{ChernTenenblat}. However, $F$ is of second order only when $H=L=0$, as  in Theorem 2.4 in \cite{ChernTenenblat}.

\begin{Lem}\label{LemHL=0}Let $f_{ij}$, $1\leqslant i \leqslant 3$, $1\leqslant j \leqslant 2$, be differentiable functions of $z_{0}, z_{1}, z_{2}$ such that (\ref{NC1}) and (\ref{NC2}) hold and $f_{21}=\eta$ a nonzero parameter.  Suppose $f_{31} = \pm f_{11}\neq 0$. Then $z_{0,t} = F(z_{0},z_{1},z_{2})$ describes an $\eta$ pseudo-spherical surface with associated $1$-forms $\omega^i = f_{i1}dx + f_{i2}dt$, if and only if $f_{22} = \lambda$, where $\lambda$ is constant, $f_{32} = \pm f_{12}$, and 
\begin{equation}\label{EET2}
F=\dfrac{f_{12,z_{1}}}{f_{11,z_{0}}}z_{2} + \dfrac{f_{12,z_{0}}}{f_{11,z_{0}}} z_{1}\mp \dfrac{\lambda f_{11} - \eta f_{12}}{f_{11,z_{0}}}.
\end{equation}
\end{Lem}

\begin{proof}Immediate when $k=2$ in Theorem 2.4 in \cite{ChernTenenblat}.\end{proof}

\section{Proof of Theorem \ref{EvolRes}}\label{ProofEvol}

\begin{Lem}\label{LemCoeff}Let $u_t=F(u,u_x,u_{xx})$ be a second-order evolution equation describing $\eta$ pseudo-spherical surfaces as in Lemma \ref{LemHLneq0} or in Lemma \ref{LemHL=0}. If  there exists a local isometric immersion of a surface determined by a solution $u$ for which  the  coefficients of the second fundamental form (\ref{w13_w23}) depend on a jet of finite order of $u$, i.e., $a, b$ and $c$ depend on $x, t, u, \dots, \partial^\ell u/\partial x^\ell$, where $\ell$ is finite, then $a, b$ and $c$ are universal, i.e., $a, b$ and $c$ depend only on $x$ and $t$.  \end{Lem}

\begin{proof}Assume $a, b$ and $c$ depend on a jet of finite order, i.e., they depend on $x, t, z_{0}, \dots$ and $z_{\ell}$, where $\ell$ is fixed. Then (\ref{Eq1}) becomes
\begin{equation}\label{Eq1-k-1}
\begin{split}
&f_{11}a_{t} + \eta b_{t} - f_{12}a_{x} - f_{22}b_{x} - 2b(f_{11}f_{32} - f_{31}f_{12}) + (a-c)(\eta f_{32} - f_{22}f_{31}) \\ & - \sum_{i=0}^\ell (f_{12}a_{z_{i}} + f_{22}b_{z_{i}})z_{i+1} + \sum_{i=0}^\ell 
(f_{11}a_{z_{i}} + \eta b_{z_{i}})z_{i,t} = 0,\end{split}
\end{equation}
and (\ref{Eq2}) becomes
\begin{equation}\label{Eq2-k-1}
\begin{split}
&f_{11}b_{t} + \eta c_{t} - f_{12}b_{x} - f_{22}c_{x} +(a-c)(f_{11}f_{32} - f_{12}f_{31}) + 2b(\eta f_{32} - f_{22}f_{31})\\& - \sum_{i=0}^\ell (f_{12}b_{z_{i}} + f_{22}c_{z_{i}})z_{i+1} + \sum_{i=0}^\ell 
(f_{11}b_{z_{i}} + \eta c_{z_{i}})z_{i,t} = 0.\end{split}
\end{equation}
Since $f_{22,z_{0}} \neq 0$ and $f_{11,z_{0}} \neq 0$ for evolution equations (\ref{EET1}), and $f_{11,z_{0}} \neq 0$ and $f_{12,z_{1}} \neq 0$ for evolution equations (\ref{EET2}), differentiating (\ref{Eq1-k-1}) and (\ref{Eq2-k-1}) with respect to $z_{\ell+2}$ leads to  
$f_{11}a_{z_{\ell}} + \eta b_{z_{\ell}} = f_{11}b_{z_{\ell}} + \eta c_{z_{\ell}} = 0$, and hence
\begin{equation}\label{bzkczk}
b_{z_{\ell}} = -\dfrac{f_{11}}{\eta}a_{z_{\ell}}, \quad \text{ and } \quad c_{z_{\ell}} = \dfrac{f_{11}^2}{\eta^2}a_{z_{\ell}}.
\end{equation}

Differentiating  the Gauss equation (\ref{Gauss}) with respect to $z_\ell$ leads to  $ca_{z_{\ell}} + ac_{z_{\ell}} - 2bb_{z_{\ell}} = 0$, and substituting (\ref{bzkczk}) in the latter  leads  to 
\begin{equation}
\bigg[c+ \bigg(\dfrac{f_{11}}{\eta}\bigg)^2a + 2\dfrac{f_{11}}{\eta}b \bigg]a_{z_{\ell}}=0.
\end{equation}

If $c+ \bigg(\dfrac{f_{11}}{\eta}\bigg)^2a + 2\dfrac{f_{11}}{\eta}b = 0$ on an open set, then substituting the expression of $c$ in the Gauss equation $-ac+b^2 = 1$ leads to $ (f_{11}a/ \eta + b )^2 = 1$, so that 
\begin{equation*}
b = \pm 1 - \dfrac{f_{11}}{\eta}a, \quad \text{ and } \quad c = \bigg(\dfrac{f_{11}}{\eta}\bigg)^2 a \mp 2\dfrac{f_{11}}{\eta}.
\end{equation*}
We have then
\begin{eqnarray*}
&D_t b  = -\dfrac{f_{11}}{\eta}D_t a  - \dfrac{a}{\eta}f_{11,z_0} F, \quad &D_t c = \bigg(\dfrac{f_{11}}{\eta}\bigg)^2 D_t a + \dfrac{2}{\eta}\bigg( \dfrac{f_{11}}{\eta}a \mp 1\bigg) f_{11,z_0}F,\\
& D_x b  = -\dfrac{f_{11}}{\eta}D_x a  - \dfrac{a}{\eta}f_{11,z_0} z_1,, \quad 
 & D_x c = \bigg(\dfrac{f_{11}}{\eta}\bigg)^2 D_x a + \dfrac{2}{\eta}\bigg( \dfrac{f_{11}}{\eta}a \mp 1\bigg) f_{11,z_0}z_1,
\end{eqnarray*}
and hence
\begin{eqnarray}
f_{11}D_t a + \eta D_t b & = & -af_{11,z_0}F,\label{1em4}\\
f_{11}D_t b + \eta D_t c & = & \bigg(\dfrac{f_{11}}{\eta} a \mp 2\bigg)f_{11,z_0}F,\label{2em4}\\
f_{12}D_x a + f_{22}D_x b & = & -\dfrac{\Delta_{12}}{\eta}D_xa - \dfrac{af_{22}}{\eta}f_{11,z_0}z_1,\label{3em4}\\
f_{12}D_x b + f_{22}D_x c & = & \dfrac{f_{11}}{\eta}\dfrac{\Delta_{12}}{\eta}D_x a  + \dfrac{\Delta_{12}}{\eta^2} a f_{11,z_0}z_1 + \dfrac{f_{22}}{\eta}\bigg(\dfrac{f_{11}}{\eta} a \mp 2\bigg)f_{11,z_0}z_1.
\label{4em4}\end{eqnarray}
where $\Delta_{12}=f_{11}f_{22}-\eta f_{12}$. Substituting the latter four equalities in (\ref{Eq1}) lead to 
\begin{eqnarray*}
-af_{11,z_0}F  +\dfrac{\Delta_{12}}{\eta}D_xa + \dfrac{af_{22}}{\eta}f_{11,z_0}z_1- 2b (f_{11}f_{32} - f_{31}f_{12}) + (a-c)(\eta f_{32} - f_{31}f_{22}) = 0,
\end{eqnarray*}
which is equivalent to 
\begin{eqnarray}\label{eq1-1}
-af_{11,z_0}F  +\dfrac{\Delta_{12}}{\eta}\sum_{i=0}^\ell a_{z_i}z_{i+1} + \dfrac{af_{22}}{\eta}f_{11,z_0}z_1- 2b (f_{11}f_{32} - f_{31}f_{12}) + (a-c)(\eta f_{32} - f_{31}f_{22}) = 0.
\end{eqnarray} 
Substituting the four equalities (\ref{1em4})-(\ref{4em4}) into (\ref{Eq2}) lead to
\begin{equation*}
\begin{split}
&\bigg(\dfrac{f_{11}}{\eta} a \mp 2\bigg)f_{11,z_0}F -  \dfrac{f_{11}}{\eta}\dfrac{\Delta_{12}}{\eta}D_x a  -  \dfrac{\Delta_{12}}{\eta^2} a f_{11,z_0}z_1- \dfrac{f_{22}}{\eta}\bigg(\dfrac{f_{11}}{\eta} a \mp 2\bigg)f_{11,z_0}z_1\\ &  + (a-c)(f_{11}f_{32} - f_{31}f_{12}) + 2b(\eta f_{32} - f_{31}f_{22}) = 0, 
\end{split}
\end{equation*}
which is equivalent to 
\begin{equation}\label{eq2-1}
\begin{split}
&\bigg(\dfrac{f_{11}}{\eta} a \mp 2\bigg)f_{11,z_0}F -  \dfrac{f_{11}}{\eta}\dfrac{\Delta_{12}}{\eta}\sum_{i=0}^\ell a_{z_i}z_{i+1}  -  \dfrac{\Delta_{12}}{\eta^2} a f_{11,z_0}z_1- \dfrac{f_{22}}{\eta}\bigg(\dfrac{f_{11}}{\eta} a \mp 2\bigg)f_{11,z_0}z_1\\ &  + (a-c)(f_{11}f_{32} - f_{31}f_{12}) + 2b(\eta f_{32} - f_{31}f_{22}) = 0.
\end{split}
\end{equation}
\begin{itemize}
\item If $\ell \geq 2$, then differentiating (\ref{eq1-1}) with respect to $z_{\ell + 1}$ leads to $\Delta_{12}a_{z_\ell} = 0$. Thus  $a_{z_{\ell}} = 0$ and also $b_{z_\ell} = c_{z_\ell} = 0$. 
\item If $\ell = 1$, then differentiating (\ref{eq1-1})  and  (\ref{eq2-1}) with respect to $z_2$ lead to 
\begin{eqnarray*}
-af_{11,z_0}F_{z_2} + \dfrac{\Delta_{12}}{\eta} a_{z_{1}} = 0,\\
\bigg(\dfrac{f_{11}}{\eta} a \mp 2\bigg)f_{11,z_0}F_{z_2} - \dfrac{f_{11}}{\eta}\dfrac{\Delta_{12}}{\eta}a_{z_{1}} = 0.
\end{eqnarray*}
The latter system leads  to $f_{11,z_0}F_{z_{2}} = 0$, which runs into a contradiction. 
\item If $\ell =0$, then differentiating (\ref{eq1-1})  and  (\ref{eq2-1}) with respect to $z_2$ lead to 
\begin{eqnarray*}
-af_{11,z_0}F_{z_2}  = 0,\\
\bigg(\dfrac{f_{11}}{\eta} a \mp 2\bigg)f_{11,z_0}F_{z_2}  = 0.
\end{eqnarray*}
The latter system leads  to $f_{11,z_0}F_{z_{2}} = 0$, which runs into a contradiction. 
\end{itemize}
Therefore, for all $\ell$,  (\ref{Eq1}), (\ref{Eq2}) and the Gauss equation is an inconsistent system. 

\item If $c+ \bigg(\dfrac{f_{11}}{\eta}\bigg)^2a + 2\dfrac{f_{11}}{\eta}b \neq  0$, then $a_{z_\ell} = 0$, and hence $b_{z_\ell} = c_{z_\ell}=0$, and successive differentiating leads to $a_{z_i} = b_{z_i} = c_{z_i} = 0$ for all $i=0, \dots, \ell$. 

Finally, if the functions $a, b$ and $c$ depend on a jet of finite order, then there are universal, i.e., they are functions of $x$ and $t$ only. \end{proof}

\begin{Prop}For the second-order evolution equations which describe $\eta$ pseudo-spherical surfaces as in Lemma \ref{LemHLneq0}, there  is no local isometric immersion in $\mathbb{R}^3$ of a pseudo-spherical surface determined by a solution $u$, for which  the  coefficients $a,b,c$ of the second fundamental form depend on a jet of finite order of $u$.
\end{Prop}
\begin{proof}Let $a$, $b$, and $c$ be coefficients of the second fundamental form satisfying the Gauss equation $ac-b^2=-1$. By Lemma \ref{LemCoeff}, if $a, b$ and $c$ depend on a jet of finite order, then $a, b$ and $c$ depend only on $x$ and $t$. From (\ref{f31-1})-(\ref{f32-1}),  we have
\begin{eqnarray}\label{det1}
f_{11}f_{32} - f_{12}f_{31} &=& f_{22,z_{0}}z_{1},\\\label{det2}
\eta f_{32} - f_{22}f_{31} &=& \mp \dfrac{\alpha f_{22,z_{0}}}{\sqrt{1-\alpha^2}}z_{1}.
\end{eqnarray}
Taking into account the expressions (\ref{det1}), (\ref{det2}) and (\ref{f12-1}), equations  (\ref{Eq1}) and (\ref{Eq2}) become 
\begin{eqnarray}\label{Eq1-1}
f_{11}a_{t} + \eta b_{t} - \bigg( \dfrac{f_{11}f_{22}}{\eta} \mp \dfrac{f_{22,z_{0}}}{\eta\sqrt{1-\alpha^2}}z_{1}\bigg) a_{x} - f_{22}b_{x} - 2bf_{22,z_{0}}z_{1} \mp (a-c)\dfrac{\alpha f_{22,z_{0}}}{\sqrt{1-\alpha^2}}z_{1} =0,\\\label{Eq2-1}
f_{11}b_{t} + \eta c_{t} - \bigg( \dfrac{f_{11}f_{22}}{\eta} \mp \dfrac{f_{22,z_{0}}}{\eta\sqrt{1-\alpha^2}}z_{1}\bigg) b_{x} - f_{22}c_{x} +(a-c)f_{22,z_{0}}z_{1} \mp 2b\dfrac{\alpha f_{22,z_{0}}}{\sqrt{1-\alpha^2}}z_{1} =0.
\end{eqnarray}
Differentiating (\ref{Eq1-1}) and (\ref{Eq2-1}) with respect to $z_{1}$ and the fact that $f_{22,z_{0}} \neq 0$ lead to 
\begin{equation*}
\left(\begin{array}{c}a_{x} \\b_{x}\end{array}\right) = \left(\begin{array}{cc} \pm \eta \sqrt{1-\alpha^2} & \alpha\eta \\\alpha\eta &\mp \eta \sqrt{1-\alpha^2}  \end{array}\right)\left(\begin{array}{c}2b \\a-c\end{array}\right).
\end{equation*}
The determinant of the $2\times 2$ matrix appearing in the above equation is non-zero, therefore, $a_{x}$ and $b_{x}$ can not vanish simultaneously. Otherwise, $a-c = b = 0$, and this contradicts the Gauss equation. (\ref{Eq1-1}) and (\ref{Eq2-1}) become then
\begin{eqnarray}\label{Eq1-2}
\eta f_{11}a_{t} + \eta^2 b_{t} -  f_{11}f_{22}a_{x} - \eta f_{22}b_{x} =0,\\ \label{Eq2-2}
\eta f_{11}b_{t} + \eta^2 c_{t} - f_{11}f_{22}  b_{x} -\eta f_{22}c_{x} =0.
\end{eqnarray}

Differentiating (\ref{Eq1-2}) and (\ref{Eq2-2})  with respect to $z_{0}$, and dividing by $\eta f_{11,z_{0}}$ lead to 
\begin{eqnarray}\label{Eq1-3}
a_{t} = \dfrac{(f_{11}f_{22})_{z_{0}}}{\eta f_{11,z_{0}}} a_{x} +\dfrac{f_{22,z_{0}}}{f_{11,z_{0}}} b_{x},\\\label{Eq2-3}
b_{t} = \dfrac{(f_{11}f_{22})_{z_{0}}}{\eta f_{11,z_{0}}} b_{x} +\dfrac{f_{22,z_{0}}}{f_{11,z_{0}}} c_{x}.
\end{eqnarray}
Observe that $\dfrac{(f_{11}f_{22})_{z_{0}}}{f_{11,z_{0}}}$ and $\dfrac{f_{22,z_{0}}}{f_{11,z_{0}}}$ cannot both be constant. Otherwise, $f_{22,z_{0}}=0$ which is a contradiction. Differentiating  (\ref{Eq1-3}) and (\ref{Eq2-3}) with respect to $z_{0}$ leads to 
\begin{eqnarray*}
\bigg( \dfrac{(f_{11}f_{22})_{z_{0}}}{\eta f_{11,z_{0}}}\bigg)_{z_{0}} a_{x} +\bigg(\dfrac{f_{22,z_{0}}}{f_{11,z_{0}}}\bigg)_{z_{0}} b_{x} = 0,\\
\bigg(\dfrac{(f_{11}f_{22})_{z_{0}}}{\eta f_{11,z_{0}}}\bigg)_{z_{0}} b_{x} +\bigg(\dfrac{f_{22,z_{0}}}{f_{11,z_{0}}} \bigg)_{z_{0}}c_{x} = 0.
\end{eqnarray*}
 We conclude that 
\begin{equation}\label{det}
a_{x}c_{x} - b_{x}^2 =0.
\end{equation}

Subtracting (\ref{Eq2-3}) multiplied by $a_x$ from  (\ref{Eq1-3})  multiplied by $b_x$, it follows from (\ref{det}) that  
\begin{equation}\label{axbt}
a_{x}b_{t} - a_{t}b_{x} = 0. 
\end{equation}
From (\ref{Eq1-2}), we have $f_{11}(\eta a_{t} - f_{22}a_{x}) + \eta^2 b_{t} - \eta f_{22}b_{x} =0$. Note that $(\eta a_{t} - f_{22}a_{x}) \neq 0$. Otherwise, since $f_{22,z_0}\neq 0$, we have  $a_{x} = a_{t}= 0$ and hence it follows from (\ref{det}) that $b_x=0$, which runs into a contradiction.  Therefore, 
\begin{equation}\label{EqLast}
f_{11} = \dfrac{\eta (\eta b_{t} - f_{22}b_{x})}{\eta a_{t} - f_{22}a_{x}}.
\end{equation}
Differentiating (\ref{EqLast}) with respect to $z_{0}$ and taking into account (\ref{axbt}) lead to $f_{11,z_{0}} = 0$, which is a contradiction. \end{proof}%

\begin{Prop}\label{propabcexpluniversal} Let $u_t=F(u,u_x,u_{xx})$ be a second-order evolution equation which describes \linebreak $\eta$ pseudo-spherical surfaces, as in Lemma \ref{LemHL=0}.  
There exists a local isometric immersion in $\mathbb{R}^3$  of a pseudo-spherical surface, determined by a 
solution $u$, for which the coefficients of the second fundamental form (\ref{w13_w23}) depend on a jet of finite order of $u$ if, and only if, the coefficients are  universal and are given by 
\begin{eqnarray}\label{auniversal}
a&=& \sqrt{le^{\pm2(\eta x + \lambda t )} - \gamma^2e^{\pm4(\eta x + \lambda t )} - 1}, \\\label{buniversal}
b&=& \gamma e^{\pm2(\eta x + \lambda t)}, \\\label{cuniversal}
c&=& \dfrac{\gamma^2e^{\pm4(\eta x + \lambda t)} - 1}{ \sqrt{le^{\pm2(\eta x + \lambda t )} - \gamma^2e^{\pm 4(\eta x + \lambda t )} - 1}}, 
\end{eqnarray}
$l,\gamma \in \mathbb{R}, \, l>0$   and   $l^2 >4\gamma^2$.  The 1-forms are defined on a strip of $\mathbb{R}$ where 
\begin{equation}\label{strip}
\log\sqrt{\dfrac{l-\sqrt{l^2-4\gamma^2}}{2\gamma^2}}< \pm( \eta x+\lambda t)< \log\sqrt{\dfrac{l+\sqrt{l^2-4\gamma^2}}{2\gamma^2}}.   
 \end{equation}
 Moreover, the constants $l$ and $\gamma$ have to be chosen so that the strip intersects the domain of the solution of the evolution equation.   
\end{Prop}

\begin{proof}As for the previous proposition, if $a, b$ and $c$ depend on a jet of finite order, it follows from  Lemma 4 that $a, b$ and $c$ depend only on $x$ and $t$.  We assume also that $f_{12,z_{1}} \neq 0$, otherwise, the evolution equation is not of second-order.  Equations (\ref{Eq1}) and (\ref{Eq2}) become 
 \begin{eqnarray}\label{Eq1k=02}
f_{11}a_{t} + \eta b_{t} - f_{12}a_{x}  -\lambda b_{x}  \pm (\eta f_{12} - \lambda f_{11})(a-c) = 0,\\ \label{Eq2k=02}
f_{11}b_{t} + \eta c_{t} - f_{12}b_{x}  -\lambda c_{x}  \pm (\eta f_{12} - \lambda f_{11})2b = 0.
\end{eqnarray}
Differentiating (\ref{Eq1k=02}) and (\ref{Eq2k=02}) with respect to $z_{1} $, and the fact that $f_{12,z_{1}} \neq 0$ lead to 
\begin{eqnarray}\label{Cax=0}
a_{x} \mp \eta (a-c) = 0,\\\label{Cbx=0}
b_{x} \mp 2\eta b = 0.
\end{eqnarray}
Taking into account (\ref{Cax=0}) and (\ref{Cbx=0}) and differentiating (\ref{Eq1k=02}) and (\ref{Eq2k=02}) with respect to $z_{0}$ leads to
\begin{eqnarray}\label{Cat=0}
a_{t} \mp \lambda (a-c) = 0,\\\label{Cbt=0}
b_{t} \mp 2\lambda b = 0,
\end{eqnarray}
and hence, (\ref{Eq1k=02}) and (\ref{Eq2k=02}) become
\begin{eqnarray}\label{btbx=0}
\eta b_{t} - \lambda b_{x} = 0,\\\label{ctcx=0}
\eta c_{t} - \lambda c_{x} = 0.
\end{eqnarray}
Note that (\ref{Cbx=0}) and (\ref{Cbt=0}) imply (\ref{btbx=0}), and (\ref{Cax=0}) and (\ref{Cat=0}) imply 
\begin{eqnarray}\label{atbx=0}
\eta a_{t} - \lambda a_{x} = 0,
\end{eqnarray}
and hence imply  (\ref{ctcx=0}). From (\ref{Cbx=0}) and (\ref{Cbt=0}), we conclude that 
\begin{equation}\label{b}
b= \gamma e^{\pm2(\eta x + \lambda t)}, \quad \gamma \in \mathbb{R}.
\end{equation}

Note that $a\neq 0$. Otherwise, if $a=0$, then (\ref{Cax=0}) implies that $c=0$ and the Gauss equation leads to $b=\pm1$ which contradicts (\ref{Cbx=0}). Therefore, from the Gauss equation  we have $c={(b^2-1)}{a^{-1}}$. Then, in view of (\ref{b}),  
equations (\ref{Cax=0}) and (\ref{Cat=0}) reduce to 
\begin{eqnarray*}
aa_x\mp \eta (a^2 - \gamma^2e^{\pm 4(\eta x + \lambda t)} + 1) = 0,\\
aa_t \mp\lambda( a^2  -  \gamma^2e^{\pm4(\eta x + \lambda t)} + 1) = 0.\\
\end{eqnarray*}
The latter system leads then to 
\begin{equation*}
a= \sqrt{le^{\pm2(\eta x + \lambda t )} - \gamma^2e^{\pm4(\eta x + \lambda t )} - 1}, \quad l\in \mathbb{R},
\end{equation*}
which is defined wherever $le^{\pm2(\eta x + \lambda t )} - \gamma^2e^{\pm4(\eta x + \lambda t )} - 1>0$. 
Hence $l>0$ and 
\[ 
\dfrac{l-\sqrt{l^2-4\gamma^2}}{2\gamma^2} < e^{\pm2(\eta x + \lambda t )}< \dfrac{l+\sqrt{l^2-4\gamma^2}}{2\gamma^2},
\]
{\it i.e.}, $a$ is defined on the strip described by (\ref{strip}).  
Now,  from either (\ref{Cax=0}) or (\ref{Cat=0}), we obtain 
\begin{equation*}
c= \dfrac{\gamma^2e^{\pm4(\eta x + \lambda t)} - 1}{ \sqrt{le^{\pm2(\eta x + \lambda t )} - \gamma^2e^{\pm4(\eta x + \lambda t )} - 1}}. 
\end{equation*}
 A straightforward computation shows that the converse holds. 
Finally, we observe that given a solution of the evolution equation, in order to have an immersion, one has to choose the constants $l$ and $\gamma$, such that the strip (\ref{strip}) intersects the domain
of the solution in $\mathbb{R}^2$. \end{proof}

\section{Proof of Theorem \ref{HyperbRes}}\label{HyperbProof}
We begin by introducing some notations. Given a differentiable function $u(x,t)$, we denote its partial derivatives by 
\begin{equation}\label{ziwi}
z_{i} = \dfrac{\partial^{i} u}{\partial x^{i}}, \quad w_{i} = \dfrac{\partial^{i} u}{\partial t^{i}}, 
\quad \mbox{ where } \quad  z_{0}=w_{0}=u.
\end{equation}
 We have therefore
\begin{equation*}
z_{i,x} = z_{i+1}, \quad 
z_{i,t} = \dfrac{\partial^{i-1}u_{xt}}{\partial x^{i-1}}, \quad w_{i,x} = \dfrac{\partial^{i-1} u_{xt}}{\partial t^{i-1}}, \quad w_{i,t} = w_{i+1},
\end{equation*}
and the total derivatives of a differentiable function $\varphi =\varphi(x,t,z_{0}, z_{1},w_{1}, \dots,  z_{\ell}, w_{\ell})$ are given by
\begin{eqnarray} \label{TotDerx}
D_{x}\varphi &=& \varphi_{x} +\sum_{i=0}^\ell\varphi_{z_{i}}z_{i+1}+\sum_{i=1}^\ell\varphi_{w_{i}}w_{i,x},\\\label{TotDert}
D_{t}\varphi &=& \varphi_{t} +\sum_{i=1}^\ell\varphi_{z_{i}}z_{i,t}+\sum_{i=0}^\ell\varphi_{w_{i}}w_{i+1}.
\end{eqnarray}

We also introduce the notation
\begin{equation}\label{deltaij}
\Delta_{ij} = f_{i1}f_{j2} - f_{j1}f_{i2}.
\end{equation}
Observe that 
\begin{equation}\label{deltaijneq}
\Delta_{12}\neq 0,\qquad \Delta_{13}^2+\Delta_{23}^2\neq 0.
\end{equation}
In fact, $\Delta_{12}\neq 0$ is equivalent to $\omega^{1}\wedge \omega^{2}\neq 0$. Moreover, 
$\omega^{1}\wedge \omega^{3}=\Delta _{13} dx\wedge dt $ and   $\omega^{2}\wedge \omega^{3}=\Delta_{23}dx\wedge dt$. If $\Delta_{13}=\Delta_{23}=0$, then it follows from (\ref{struct}) that $d\omega^1=d\omega^2=0$. Therefore, 
$\omega^3(e_1)=\omega^3(e_2)=0$ and hence $\omega^3=0$ that is in contradiction with 
$d\omega^3=\omega^1\wedge\omega^2$. 

The classification theorem of Rabelo and Tenenblat (see Theorem \ref{RTClassification}) for hyperbolic equations describing $\eta$ pseudo-spherical surfaces makes use of a number of lemmas. Its proof also  provides 
the coefficients $f_{ij}$ of the 1-forms (\ref{forms}) for each equation of Theorem 2. 
 We will need the lemmas and these coefficients for the proof of Theorem \ref{HyperbRes}. We therefore recall them from \cite{RabeloTenenblat90}  without proof. However, the reader can easily check, in each case stated in Lemmas \ref{LemF}-\ref{LemLin}, that the structure 
 equations (\ref{struct}) hold if, and only if, the corresponding hyperbolic equation holds. 

\begin{Lem} \cite{RabeloTenenblat90}
Let $u_{xt} = F(u,u_x)$ be a differential equation describing $\eta$ pseudo-spherical surfaces, with associated one-forms $\omega_{i} = f_{i1}dx + f_{i2}dt $, where $f_{ij}$ and $F$ are real differentiable $(\mathcal{C}^\infty$) functions on a open connected set $U\subset \mathbb{R}^2$. Then 
\begin{eqnarray*}
f_{11,u} \equiv f_{31,u} \equiv 0,&\\
f_{12,u_x} \equiv f_{22, u_x} \equiv f_{32,u_x}\equiv 0,&\\
f_{11,u_x}^2 + f_{31, u_x}^2 \neq 0,& \quad \text{ in } U.
\end{eqnarray*}
\end{Lem}


\begin{Lem}\label{LemF} \cite{RabeloTenenblat90} 
 The coefficients $f_{ij}$ of the 1-forms (\ref{forms}) for the equation 
\begin{equation}\label{eqF}
u_{xt}=F(u), \mbox{ \quad where \quad  }  F''(u)+\alpha F(u)=0, \qquad \alpha \in \mathbb{R}\setminus \{0\}
\end{equation}
are given by 
\begin{equation}\label{fijF''}
\left(\begin{array}{cc} f_{11} & f_{12}  \\ f_{21} & f_{22}  \\ f_{31} &f_{32}  \end{array}\right) = \left(\begin{array}{cc}    -\alpha (Bz_1 - AQ)&  A\alpha (QF' - \eta F)/(Q^2\alpha + \eta^2)  \\  \eta  & (\eta F' + \alpha QF)/(Q^2\alpha + \eta^2)  \\ -\alpha (Az_1 - BQ)&  B\alpha (QF' - \eta F)/(Q^2\alpha + \eta^2)   \end{array}\right), 
\end{equation}
where  $z_1=u_x$,  $A,B, Q\in \mathbb{R}$ are such that $\alpha=1/(A^2-B^2)$, $A^2-B^2\neq 0$ and $Q^2\alpha+\eta^2\neq 0$ and  
$\eta\in \mathbb{R}\setminus\{0\}$. In particular,  
if $B=0$ and hence $A\neq 0$, one has  $\alpha=1/A^2>0$ and     
\begin{equation}\label{fijF''QA}
\left(\begin{array}{cc} f_{11} & f_{12}  \\ f_{21} & f_{22}  \\ f_{31} &f_{32}  \end{array}\right) = \left(\begin{array}{cc}   \alpha AQ & \alpha A (QF'-\eta  F)/ (Q^2\alpha+\eta^2)  \\  \eta  &  (\eta F'+Q\alpha F)/(Q^2\alpha+\eta^2)  \\ -\alpha Az_1 &  0   \end{array}\right). 
\end{equation}
\end{Lem}

\begin{Lem}\label{Lemexp} \cite{RabeloTenenblat90} The coefficients $f_{ij}$ of the 1-forms (\ref{forms}) for the equation 
\begin{equation}\label{eqexp}
u_{xt}= \nu e^{\delta u}\sqrt{\beta+\gamma u_x^2}, \quad \mbox{ where\quad} \delta,\gamma,\nu\in \mathbb{R}\setminus \{0\} \mbox{ and }\quad  \beta=0, \mbox{ when } \gamma=1, 
\end{equation}
are given as follows:\\ 
a) If $\gamma\neq 1$, then 
\begin{equation}\label{fijexp}
\left(\begin{array}{cc} 
f_{11} & f_{12}  \\ f_{21} & f_{22}  \\ f_{31} &f_{32}  
\end{array}\right) = 
\left(\begin{array}{cc}  
\eta A\delta - (Bz_1\mp A\sqrt{\Delta})\delta^2/(\gamma - 1)  &   \pm A\delta \nu e^{\delta z_0} \\  
\eta  &  \pm \nu e^{\delta z_0} \\ 
\eta B\delta - (Az_1\mp B\sqrt{\Delta})\delta^2/(\gamma - 1)  &   \pm B\delta \nu e^{\delta z_0} \end{array}\right), 
\end{equation}
where $z_0=u$, $z_1=u_x$,  $\Delta = \beta+ \gamma z_1^2 >0$, $A^2-B^2 = (\gamma - 1)/\delta^2$ and $\eta\in \mathbb{R}\setminus\{0\}$. \\
 b) If $\gamma=1$,  
\begin{equation}\label{fijexp1}
\left(\begin{array}{cc} f_{11} & f_{12}  \\ f_{21} & f_{22}  \\ f_{31} &f_{32}  \end{array}\right) = \left(\begin{array}{cc}  \frac12(\frac1A + \delta^2 A)z_1 + \eta \delta A  & \pm A\delta \nu e^{\delta z_0}   \\  \eta  & \pm\nu e^{\delta z_0}  \\ \frac12(-\frac1A + \delta^2 A)z_1 \pm \eta \delta A  &  A\delta \nu e^{\delta z_0}    \end{array}\right),
\end{equation}
where $A, \eta\in \mathbb{R}\setminus\{0\}$. 
\end{Lem}

\begin{Lem}\label{LemLin} \cite{RabeloTenenblat90} The coefficients $f_{ij}$ of the 1-forms (\ref{forms}) for the equation 
\begin{equation}\label{eqLin}
u_{xt}= \lambda u+\xi u_x +\tau, \qquad \lambda,\xi,\tau \in \mathbb{R}
\end{equation} 
are given as follows:\\ 
a) If $\lambda=\xi=\tau=0$, then 
\begin{equation}\label{fijLin1}
\left(\begin{array}{cc}
f_{11} & f_{12} \\f_{21} & f_{22} \\f_{31} & f_{32}
\end{array}\right)=
\left(\begin{array}{cc} z_{1} &  0 \\ 
\eta & e^{z_0}  \\  
 \eta &  e^{z_0} \end{array}\right),
 \end{equation}
 where $z_0=u$, $z_1=u_x$ and $\eta\neq 0$.
\\
b) If $\lambda\neq 0$, then
\begin{equation}\label{fijLin2}
\left(\begin{array}{cc}
f_{11} & f_{12} \\f_{21} & f_{22} \\f_{31} & f_{32}
\end{array}\right)=
\left(\begin{array}{cc} \pm \eta Tz_{1}/\lambda &  Tz_{0} + \tau T/\lambda \\ 
\eta & \lambda/\eta \mp \xi \\  
 \eta Tz_{1}/\lambda &  \pm Tz_{0} \pm \tau T/\lambda \end{array}\right)
\end{equation}
where $T,\eta\in \mathbb{R}\setminus\{0\}$.\\ 
c) If $\lambda=0$ and $\xi^2+\tau^2\neq 0$,  then  
\begin{equation}\label{fijLin3}
\left(\begin{array}{cc}f_{11} & f_{12} \\f_{21} & f_{22} \\f_{31} & f_{32}\end{array}\right)=\left(\begin{array}{cc} \int dz_1/F(z_1)&  1/\eta\\ \eta & 0 \\   \int dz_1/F(z_1)&  1/\eta \end{array}\right)
\end{equation}
where $\eta\in \mathbb{R}\setminus\{0\}$. 
\end{Lem}

 Having recalled these results from \cite{RabeloTenenblat90}, we are now ready to proceed with the proof of Theorem \ref{HyperbRes}. 
The proof  consists of a number of technical lemmas and propositions, in which we analyze the existence of solutions  for the system of equations (\ref{Eq1}), (\ref{Eq2}) and (\ref{Gauss}) that depend on $u$ and finitely many derivatives, for each of  the  classes of hyperbolic equations obtained by Rabelo and Tenenblat in Theorem \ref{RTClassification}.

With the notation introduced in (\ref{deltaij}), equations (\ref{Eq1}) and (\ref{Eq2}) are written as  
\begin{eqnarray}\label{EQ1}
&& f_{11}D_ta + \eta D_tb - f_{12}D_xa - f_{22}D_xb - 2b\Delta_{13} + (a-c)\Delta_{23} = 0,\\
\label{EQ2}
&& f_{11}D_tb + \eta D_tc - f_{12}D_xb - f_{22}D_xc  + (a-c)\Delta_{13} + 2b\Delta_{23} = 0.
\end{eqnarray}

\begin{Lem}\label{Lem_ac}
Consider an equation $u_{xt}=F(u,u_x)$ describing $\eta$ pseudo-spherical surfaces, with 1-forms 
$\omega^i$ as in (\ref{forms}) where the functions $f_{ij}$ are given by (\ref{fijF''})-(\ref{fijLin3}). Assume there is a local isometric immersion of 
any pseudo-spherical surface, determined by a solution $u(x,t)$, for which the coefficients $a,\, b, \,c$ of the forms 
$\omega_1^3$ and $\omega_2^3$ depend on a jet of finite order of $u$. Then 
\begin{enumerate}
\item [i)] $a\neq 0$ on any open set.
\item [ii)] $c=0$ on an open set $U$ if, and only if, $f_{11}=0$ on $U$, i.e., $F$ satisfies (\ref{eqF}) and  $f_{ij}$ are given by (\ref{fijF''QA}) with $Q=0$.  
In this case, $\alpha=1/A^2>0$,   
\begin{equation}\label{abcf110}
a=\pm\dfrac{2}{A\alpha}  \dfrac{F'}{F},\quad  b=\pm 1,  \quad \text{ and } \quad c=0.
\end{equation}
\end{enumerate}
\end{Lem}

\begin{proof} If there is a local isometric immersion of the  pseudo-spherical surface, then 
(\ref{EQ1}), (\ref{EQ2}) and (\ref{Gauss}) must be satisfied by $a$, $b$ and $c$.   

\vspace{.1in}

 i) Assume  $a=0$ on an open set, then it follows from (\ref{Gauss}) that $b\pm 1$. 
Substituting into (\ref{EQ1}) and (\ref{EQ2}) leads to 
\begin{eqnarray}
&& \mp 2\Delta_{13}-c\Delta_{23}=0, \label{a1}\\
&& \eta D_t c-f_{22} D_x c -c\Delta_{13}\pm 2\Delta_{23}=0.\label{a2}
\end{eqnarray}
It follows from (\ref{a1}) and (\ref{deltaijneq}) that $\Delta_{23}\neq 0$ and 
$c=\mp 2\Delta_{13}/\Delta_{23}$. 
Since $\Delta_{13}$ and $\Delta_{23}$ depend only on $z_0$ and $z_1$, we conclude that 
$c$ depends only on $z_0$ and $z_1$ and (\ref{a2}) reduces to
\begin{equation}\label{a4}
\eta(c_{z_1}F + c_{z_0}w_1)-f_{22}(c_{z_0} z_1+c_{z_1}z_2)-c\Delta_{13}\pm 2\Delta_{23}=0.
\end{equation}
Taking the derivative of this equation with respect to $z_2$ and $w_1$ implies that 
$f_{22}c_{z_1}=0$ and $c_{z_0}=0$. If $f_{22}\neq 0$ then  $c$ is constant and 
(\ref{a4}) reduces to $-c\Delta_{13}\pm2\Delta_{23}=0$ i.e., we have 
\[
\left(\begin{array}{cc} -c &\pm 2\\ \mp 2 & -c\end{array}\right)
\left(\begin{array}{c} \Delta_{13}\\ \Delta_{23}\end{array}\right)=
\left(\begin{array}{c} 0\\ 0\end{array}\right).
\]
Since the determinant is nonzero, it implies that $\Delta_{13}=\Delta_{23}=0$ 
which contradicts (\ref{deltaijneq}). If $f_{22}=0$ on an open set, then the 
functions $f_{ij}$ are given by (\ref{fijLin3}) and hence $\Delta_{13}=0$ and 
$\Delta_{23}=1$. Then (\ref{a1}) implies that $c=0$ and (\ref{a2}) gives a contradiction.
This concludes the proof of i). 

\vspace{.1in} 

ii) Observe that except for the functions $f_{ij}$ given by (\ref{fijF''QA}) with $Q=0$, $f_{11}$ does not vanish on an open set.   
We will first show that if $f_{11}=0$ on an open set i.e, $F$ satisfies (\ref{eqF}) and $f_{ij}$ are given by (\ref{fijF''QA}) and $Q=0$, then $c=0$.  In fact, for such $f_{ij}$s  we have 
$\Delta_{13}=-A^2\alpha^2 F(u) z_1/\eta$, $\Delta_{23}= A\alpha F'(u) z_1/\eta$, $\alpha=1/A^2>0$ and $A\neq 0$. Hence (\ref{EQ1}) and (\ref{EQ2}) reduce to 
\begin{eqnarray*}
\eta D_tb -f_{12}D_xa-f_{22}D_xb-2b\Delta_{13}+(a-c)\Delta_{23}=0,\\ 
\eta D_t c-f_{12}D_xb-f_{22}D_xc+(a-c)\Delta_{13}+2b\Delta_{23}=0, 
\end{eqnarray*}
where $f_{12}=-\alpha AF/\eta$ and $f_{22}=F'/\eta$. Assume $c\neq0$, then it follows from (\ref{Gauss}) that $a=(b^2-1)/c$. Assume that $a$,$b$ and $c$ depend on a jet of order $\ell$ of $u$. For $\ell\geq 1$,  
taking derivatives of both equations with respect to $w_{\ell+1}$ implies that $b_{w_\ell}=c_{w_\ell}=0$ and 
hence $a_{w_k}=0$. Successive differentiation  with respect to $w_k$,...$w_1$ imply that $a$, $b$ and 
$c$ do not depend on $w_\ell$,...$w_0$. Successive differentiation with respect to $z_{\ell+1}$, ...$z_2$ imply 
that $a$, $b$ and $c$ do not depend on $z_\ell$,...$z_1$. Hence, $a$, $b$ and $c$ depend only on $x$ and $t$.  Therefore, the above system of equations reduce to
\begin{eqnarray*}
\eta b_t +\frac{\alpha A F}{\eta} a_x-\frac{F'}{\eta}b_x +2b\frac{\alpha^2 A^2 F}{\eta} z_1 +(a-c)\frac{\alpha AF'}{\eta}z_1=0,  \\
\eta c_t +\frac{\alpha A F}{\eta} b_x-\frac{F'}{\eta}c_x +(a-c)\frac{\alpha^2 A^2 F}{\eta} z_1 + 2b \frac{\alpha AF'}{\eta}z_1=0. 
\end{eqnarray*} 
Taking the derivative with respect to $z_1$ we get 
\begin{equation}\label{a-c}
\left( \begin{array}{cc}
2b & a-c\\ a-c & 2b 
\end{array}\right)
\left( \begin{array}{c}
\alpha A F\\ F'
\end{array}\right)  =
\left( \begin{array}{c}
0\\ 0
\end{array}\right).
\end{equation}
Since $\alpha A F$ and $F'$ are not zero we get $a-c=\pm 2b$ and the derivative with respect to 
$z_0$ of any equation of (\ref{a-c}) reduces to $b(AF'\mp F)=0$ as a consequence of (\ref{eqF}). 
If $b=0$ then Gauss equation (\ref{Gauss}) reduces to $a^2=-1$. If $F=\pm AF'$ then the derivative with respect to $z_0$ implies that $\alpha A^2=-1$. In both cases we get a contradiction. Therefore, $c=0$.    

\vspace{.05in}

 Conversely, assume $c=0$ on an open set, then (\ref{Gauss}) implies $b=\pm 1$ and  (\ref{EQ1}) and 
(\ref{EQ2}) reduce to 
\begin{eqnarray}
&& f_{11} D_t a -f_{12} D_x a \mp 2\Delta_{13}+ a \Delta_{23}=0,\label{c1}\\
&& a\Delta_{13}\pm 2 \Delta_{23}=0. \label{c2}
\end{eqnarray}
It follows from (\ref{c2}) and (\ref{deltaijneq}) that $\Delta_{13}\neq 0$ and 
$ a=\mp 2\Delta_{23}/\Delta_{13}$.   
Since $\Delta_{13}$ and $\Delta_{23}$ depend only on $z_0$ and $z_1$, we conclude that 
$a$ depends only on $z_0$ and $z_1$ and (\ref{c1}) reduces to
\begin{equation}\label{c4}
f_{11}(a_{z_1} F+a_{z_0} w_1)-f_{12}(a_{z_0}z_1+a_{z_1}z_2)\mp2\Delta_{13}+a\Delta_{23}=0.
 \end{equation}
Differentiation with respect to $w_1$ and $z_2$ implies 
\begin{equation}\label{c5}
f_{11}a_{z_0}= f_{12}a_{z_1}=0.
\end{equation} 
Since $\Delta_{12}\neq 0$, we observe that $f_{11}$ and $f_{12}$ cannot vanish simultaneously. 

If  both $f_{11}\neq 0$ and $f_{12}\neq 0$ then from (\ref{c5}) we conclude that 
$a$ is constant and (\ref{c4}) reduces to $\mp 2\Delta_{13}+ a \Delta_{23}=0$. 
This equation with (\ref{c2}) implies that $\Delta_{13}=\Delta_{23}=0$ which contradicts (\ref{deltaijneq}).

If $f_{12}=0$ on an open set,  then $f_{ij}$ are given by (\ref{fijF''}) with $A=0$, $B\neq 0$  or (\ref{fijexp}) with $A=0$, $B\neq 0$ or (\ref{fijLin1}). Since $f_{11}\neq 0$, it follows from 
(\ref{c5}) that $a_{z_0}=0$ and (\ref{c4}) reduces to 
\begin{equation}\label{c8}
f_{11}a_{z_1}F\mp 2\Delta_{13}+a\Delta_{23}=0.
\end{equation}   
If  $f_{ij}$ are given by (\ref{fijF''}) with $A=0$, $B\neq 0$, then 
\[
\Delta_{13}=\frac{\alpha(QF'-\eta F }{Q^2\alpha+\eta^2}z_1,\qquad \Delta_{23}=-B\alpha F, \qquad 
a_{z_1}z_1=-a.
\]
Substituting into (\ref{c8}) and differentiating twice with respect to $z_1$ runs into a  contradiction.
If $f_{ij}$ are given by (\ref{fijexp}) with $A=0$, $B\neq 0$, then  
\[
\Delta_{13}=\mp B\nu \delta^2 z_1 e^{\delta z_0},\qquad \Delta_{23}=\pm \nu \delta  z_1 e^{\delta z_0}, \qquad 
a=\pm \frac{2}{B\delta}.
\]
Therefore, (\ref{c8}) reduces to $\mp 2\Delta_{13}+a\Delta_{23}=0$ which is in contradiction with 
(\ref{c2}). Finally if $f_{ij}$ are given by  (\ref{fijLin1}), then $\Delta_{23}=0$, hence it follows from (\ref{deltaijneq}) and (\ref{c2}) that $a=0$
 which is a contradiction. 
 
  We conclude that if $c=0$ on an open set, then 
 $f_{11}=0$ i.e.,  $f_{ij}$ are given by (\ref{fijF''QA}) with $Q=0$. Therefore 
$\Delta_{13}=-A^2\alpha^2 F(u) z_1/\eta$, $\Delta_{23}= A\alpha F'(u) z_1/\eta$ and 
hence  (\ref{c2}) implies that $a=\pm 2 F'/(A\alpha F)$. Moreover, (\ref{c4}) is an identity since 
$A^2\alpha=1$.  This concludes the proof of Lemma \ref{Lem_ac}.\end{proof}

\vspace{.1in}

Consider an equation $u_{xt}=F(u,u_x)$ describing $\eta$ pseudo-spherical surfaces given by Lemmas 6-8.
The existence of  a local isometric immersion in $\mathbb{R}^3$ of any pseudo-spherical surface, determined by 
a solution $u$,  for which  the coefficients  
$a$, $b$ and $c$ depend on $x,t,z_0,z_1,w_1,...z_\ell,w_\ell$, is equivalent to requiring that (\ref{EQ1}), (\ref{EQ2}) and (\ref{Gauss}) must be satisfied.    
Substituting the expressions of the total derivatives with respect to $x$ and $t$ given by (\ref{TotDerx}) and (\ref{TotDert}), we rewrite  (\ref{EQ1}) and (\ref{EQ2}) as 
\begin{equation}
\begin{split}\label{Eq1k}
&f_{11}a_t + \eta b_t + \sum_{i=0}^\ell(f_{11}a_{w_i}+ \eta b_{w_i})w_{i+1} + \sum_{i=1}^\ell(f_{11}a_{z_i}+ \eta b_{z_i})\dfrac{\partial^{i-1}F}{\partial x^{i-1}} -( f_{12}a_x + f_{22} b_x) \\ &- \sum_{i=0}^\ell (f_{12}a_{z_i}+ f_{22} b_{z_i})z_{i+1} - \sum_{i=1}^\ell(f_{12}a_{w_i}+ f_{22} b_{w_i})\dfrac{\partial^{i-1}F}{\partial t^{i-1}}- 2b\Delta_{13} + (a-c)\Delta_{23} = 0,
\end{split}
\end{equation}
and 
\begin{equation}
\begin{split}\label{Eq2k}
&f_{11}b_t + \eta c_t + \sum_{i=0}^\ell(f_{11}b_{w_i}+ \eta c_{w_i})w_{i+1} + \sum_{i=1}^\ell(f_{11}b_{z_i}+ \eta c_{z_i})\dfrac{\partial^{i-1}F}{\partial x^{i-1}} -( f_{12}b_x + f_{22} c_x) \\ &- \sum_{i=0}^\ell(f_{12}b_{z_i}+ f_{22} c_{z_i})z_{i+1} - \sum_{i=1}^\ell(f_{12}b_{w_i}+ f_{22} c_{w_i})\dfrac{\partial^{i-1}F}{\partial t^{i-1}} + (a-c)\Delta_{13} + 2b\Delta_{23} = 0.
\end{split}
\end{equation}
Differentiating (\ref{Eq1k}) and (\ref{Eq2k}) with respect to $w_{\ell+1}$ leads to 
\begin{equation}\label{Cons1}
f_{11}a_{w_\ell} + \eta b_{w_\ell} = 0 \qquad f_{11}b_{w_\ell} + \eta c_{w_\ell} = 0.
\end{equation} 
Differentiation of the Gauss equation (\ref{Gauss}) with respect to $w_\ell$ gives $ca_{w_\ell} + ac_{w_\ell} - 2bb_{w_\ell} = 0$. Taking into account (\ref{Cons1}) in the latter, we obtain 
\begin{equation}\label{Eqpro}
\bigg[ c +  \bigg(\dfrac{f_{11}}{\eta}\bigg)^2a +  2\dfrac{f_{11}}{\eta} b\bigg]a_{w_\ell} = 0.
\end{equation}
The following two lemmas will consider the cases in which the expression between brackets in
(\ref{Eqpro}) vanishes or not on an open set.

\begin{Lem} \label{Claim I} 
Consider an equation $u_{xt}=F(u,u_x)$ describing $\eta$ pseudo-spherical surfaces, with 1-forms 
$\omega^i$ as in (\ref{forms}) where the functions $f_{ij}$ are given by (\ref{fijF''})-(\ref{fijLin3}). Assume there is a local isometric immersion of 
a pseudo-spherical surface determined by a solution $u(x,t)$, for which  the coefficients $a,\, b, \,c$ of the second fundamental form depend on a jet of finite order of $u$.  
 If 
\begin{equation}\label{Eqcf11}
c +  \bigg(\dfrac{f_{11}}{\eta}\bigg)^2a +  2\dfrac{f_{11}}{\eta} b = 0
\end{equation}
 on a non empty open set, then   
\begin{enumerate}
\item [i)] For equation (\ref{eqF}) with $f_{ij}$ as in  (\ref{fijF''QA})
$a,\,b$ and $c$ are given by 
\begin{equation}\label{newabc}
a= \pm\frac{2\eta}{A (Q^2\alpha+\eta^2)}\left( \frac{\eta F'}{\alpha F}+ Q \right), \qquad
b=\mp\frac{1}{ Q^2\alpha +\eta^2}\left(2\eta Q\frac{F'}{F}+ Q^2\alpha -\eta^2\right), 
\end{equation}
\[
c=\pm\frac{2QA\alpha}{ Q^2\alpha +\eta^2}\left(Q\frac{F'}{F} -\eta\right), 
\]
where $\alpha=1/A^2$. In particular when $Q=0$, $a, b, c$ are given by  (\ref{abcf110}). 

\item [ii)] For all equations, except those considered in i),  equations  
(\ref{EQ1}), (\ref{EQ2}) and (\ref{Gauss}) form an inconsistent system. 
\end{enumerate}
\end{Lem}

\vspace{.1in}

\begin{proof}  
If (\ref{Eqcf11}) holds then substituting $c$ into the Gauss equation (\ref{Gauss}) leads to 
$(f_{11}a/{\eta}  + b)^2 = 1$, and hence 
\begin{equation}\label{b=c=Case1}
b= \pm 1 - \dfrac{f_{11}}{\eta}a \quad  \text{ and }  
\quad c = \bigg(\dfrac{f_{11}}{\eta}\bigg)^2a  \mp 2\dfrac{f_{11}}{\eta}.
\end{equation}

Therefore,
\begin{eqnarray*}
f_{11}D_ta + \eta D_tb &=& - af_{11, z_1}F,\\
f_{12}D_xa + f_{22} D_xb &=& -\dfrac{\Delta_{12}}{\eta}D_xa - \dfrac{af_{22}f_{11,z_1}}{\eta}z_2,\\
f_{11}D_tb + \eta D_tc &=& \dfrac{af_{11}f_{11, z_1}}{\eta}F \mp 2 f_{11,z_1}F,\\
f_{12}D_xb + f_{22} D_xc &=& \dfrac{f_{11}\Delta_{12}}{\eta^2}D_xa + \dfrac{\Delta_{12}af_{11,z_1}}{\eta^2}z_2 + \dfrac{af_{22}f_{11}f_{11,z_1}}{\eta^2} z_2 \mp 2\dfrac{f_{22}f_{11,z_1}}{\eta}z_2. 
\end{eqnarray*}
Equation (\ref{EQ1})  becomes
\begin{equation}\label{Eq1cf110}
- af_{11, z_1}F +\dfrac{\Delta_{12}}{\eta}D_xa + \dfrac{af_{22}f_{11,z_1}}{\eta}z_2  \mp 2\Delta_{13} + 2\dfrac{f_{11}}{\eta}a\Delta_{13} + \bigg[1-\bigg(\dfrac{f_{11}}{\eta}\bigg)^2\bigg]a\Delta_{23}  \pm2\dfrac{f_{11}}{\eta}\Delta_{23} = 0
\end{equation}
and (\ref{EQ2}) becomes
\begin{equation}
\begin{split}\label{Eq2cf110}
&\dfrac{af_{11}f_{11, z_1}}{\eta}F \mp 2 f_{11,z_1}F
- \dfrac{f_{11}\Delta_{12}}{\eta^2}D_xa - \dfrac{\Delta_{12}af_{11,z_1}}{\eta^2}z_2 - \dfrac{af_{22}f_{11}f_{11,z_1}}{\eta^2}z_2  \pm 2\dfrac{f_{22}f_{11,z_1}}{\eta}z_2\\
&+ \bigg[1-\bigg(\dfrac{f_{11}}{\eta}\bigg)^2\bigg]a\Delta_{13}  \pm2\dfrac{f_{11}}{\eta}\Delta_{13}  \pm 2\Delta_{23} - 2\dfrac{f_{11}}{\eta}a\Delta_{23} = 0. 
\end{split}
\end{equation}

If $\ell \geqslant 2$, then differentiating  (\ref{Eq1cf110}) with respect to $z_{\ell+1}$ leads to $a_{z_\ell} = 0$. Successive differentiation with respect to $z_\ell, ..., z_3$ leads to $a_{z_\ell} = a_{z_{\ell-1}} = \dots = a_{z_2} = 0$. If $\ell\geq 1$, then differentiating  (\ref{Eq1cf110}) and (\ref{Eq2cf110})with respect to $z_{2}$ leads to 
\begin{eqnarray*}
\Delta_{12}a_{z_1} + af_{22}f_{11,z_1} = 0,\\ 
-f_{11}\Delta_{12}a_{z_1} - \Delta_{12}af_{11, z_1} - af_{22}f_{11}f_{11,z_1} \pm 2\eta  f_{22}f_{11, z_1} = 0, 
\end{eqnarray*}
which is equivalent to 
\begin{eqnarray}\label{Diffz1Eq1-1}
\Delta_{12}a_{z_1} + af_{22}f_{11,z_1} = 0, \\\label{Diffz1Eq2-1}
(\Delta_{12}a  \mp 2\eta  f_{22}) f_{11, z_1} = 0. 
\end{eqnarray}

\vspace{.1in}

i) For equation (\ref{eqF}) with $f_{ij}$ given by (\ref{fijF''QA}) 
we have $f_{11,z_1}=0$.  Hence (\ref{Diffz1Eq2-1}) is trivially satisfied 
and (\ref{Diffz1Eq1-1}) implies that $a_{z_1}=0$.  Moreover, (\ref{Eq1cf110}) and 
(\ref{Eq2cf110}) reduce to
\begin{eqnarray}
\label{Eq1f11z10}
 \frac{\Delta_{12}}{\eta}D_xa+2\left( \frac{f_{11}}{\eta} a\mp 1\right)\Delta_{13}
 +\left[ \left(1-\frac{f_{11}^2}{\eta^2} \right) a\pm 2\frac{f_{11}}{\eta}\right]\Delta_{23},\\ 
\label{Eq2f11z10} 
 -f_{11}\frac{\Delta_{12}}{\eta^2}D_xa
 +\left[\left(1-\frac{f_{11}^2}{\eta^2} \right) a \pm 2\frac{f_{11}}{\eta}\right] \Delta_{13}
 -2\left(\frac{f_{11}}{\eta} a\mp 1\right)\Delta_{23}.
\end{eqnarray}
Adding equation (\ref{Eq1f11z10}) multiplied by $f_{11}/\eta$ with (\ref{Eq2f11z10}) and cancelling a nonzero factor, we get 
\[
a\Delta_{13}-\left(\frac{f_{11}a}{\eta}\mp 2\right)\Delta_{23}=0. 
\]
Since $\Delta_{13}-f_{11}\Delta_{23}/\eta=f_{31}\Delta_{12}/\eta$, we conclude that  $a=\mp 2\Delta_{23}\eta/(f_{31}\Delta_{12})$. For the functions $f_{ij}$ as in 
(\ref{fijF''QA}) we have $f_{31}= -\alpha Az_1\neq 0$ and 
\[
\Delta_{12}=\alpha A F, \qquad \Delta_{13}=\frac{\alpha(QF'-\eta F)}{Q^2\alpha +\eta^2}z_1, 
\qquad \Delta_{23}=\frac{\alpha A(\eta F'+\alpha Q F)}{Q^2\alpha +\eta^2}z_1.
\]
Therefore, we conclude that $a$ is given by 
\[
a= \pm\frac{2\eta}{\alpha Q^2+\eta^2}\left( \eta A \frac{F'}{F}+\frac{Q}{A} \right).
\]
 A straightforward computation shows that substituting  the expressions of $a$, $D_xa=a_{z_0}z_1$, $f_{11}=\alpha AQ$ and using the fact 
that $\alpha A^2=1$ equation  (\ref{Eq1f11z10}) is trivially satisfied. It follows from 
(\ref{b=c=Case1}) that $b$ and $c$ are given as in (\ref{newabc}). Observe that when $Q=0$ 
then (\ref{newabc} reduces to (\ref{abcf110}). 

\vspace{.1in}

ii) For all equations except those considered in i) we have $f_{11,z_1}\neq 0$. 

If $\ell=0$, then differentiating  (\ref{Eq1cf110}) and (\ref{Eq2cf110}) with respect to $z_{2}$ leads to $af_{22}f_{11,z_1} = 0$ and to $\Delta_{12}af_{11,z_1} + af_{22}f_{11}f_{11,z_1} \mp2\eta f_{22}f_{11,z_1} = 0$. From Lemma \ref{Lem_ac} $a\neq 0$, hence $f_{22}f_{11,z_1}=0$ and $\Delta_{12} f_{11,z_1}=0$. This implies that $f_{11,z_1}=0$ which is a contradiction. Therefore,  $\ell\geq 1$.

If $f_{22} =0$, which is the case for equation (\ref{eqLin}) with $f_{ij}$ given by (\ref{fijLin3}), then  (\ref{Diffz1Eq1-1}) and (\ref{Diffz1Eq2-1}) leads to $a=0$ which contradicts Lemma \ref{Lem_ac}. Thus, (\ref{EQ1}), (\ref{EQ2}), and the Gauss equation form an  inconsistent system.  

 If $f_{22}\neq 0$, (which is the case for all equations except (\ref{eqLin}) with $\lambda=0,\, \xi^2+\tau^2\neq 0$)   then dividing (\ref{Diffz1Eq2-1}) by $f_{11,z_1}$  leads to $\Delta_{12}a \mp 2 \eta f_{22} = 0$, and differentiating the latter with respect to $z_1$ gives $\Delta_{12}a_{z_1} + a \Delta_{12, z_1}=0$, where from (\ref{deltaij}) we have  $\Delta_{12,z_1} = f_{22}f_{11,z_1}$. Therefore, (\ref{Diffz1Eq1-1}) is a consequence of (\ref{Diffz1Eq2-1}). From (\ref{Diffz1Eq2-1}), we have 
\begin{equation}
a=\pm 2\eta \dfrac{f_{22}}{\Delta_{12}}, 
\end{equation}
which means that $a_{w_1} = a_{x}=a_{t} = 0$, i.e., $a$ is a function of $z_0$ and $z_1$ only. Equations (\ref{Eq1cf110}) and (\ref{Eq2cf110}) become
\begin{eqnarray*}
-af_{11,z_1}F+\dfrac{\Delta_{12}}{\eta}a_{z_0}z_1 \mp 2\Delta_{13} + 2\dfrac{f_{11}}{\eta}a\Delta_{13}+ \bigg[1-\bigg(\dfrac{f_{11}}{\eta}\bigg)^2\bigg]a\Delta_{23}  \pm2\dfrac{f_{11}}{\eta}\Delta_{23}=0, \\
\dfrac{af_{11}f_{11,z_1}F}{\eta}\mp 2f_{11,z_1}F - f_{11}\dfrac{\Delta_{12}}{\eta^2}a_{z_0}z_1 + \bigg[1-\bigg(\dfrac{f_{11}}{\eta}\bigg)^2\bigg]a\Delta_{13}
\pm 2\frac{f_{11}}{\eta} \Delta_{13} 
\pm2\Delta_{23} - 2\dfrac{f_{11}}{\eta}a\Delta_{23}=0,
\end{eqnarray*}
which are equivalent to 
\begin{eqnarray}\label{Eq1f11neq0-2}
af_{11,z_1}F-\dfrac{\Delta_{12}}{\eta}a_{z_0}z_1 -\dfrac{f_{11}}{\eta}a\Delta_{13 }   \pm 2\Delta_{13}- a\Delta_{23}= 
\frac{f_{11}}{\eta}\left( af_{31}\frac{\Delta_{12}}{\eta}\pm 2\Delta_{23}\right)
\end{eqnarray}
and
\begin{equation}\label{Eq2f11neq0-2}
\pm 2f_{11,z_1}F = \dfrac{f_{11}}{\eta} \bigg[ af_{11,z_1}F -\dfrac{\Delta_{12}}{\eta}a_{z_0}z_1 -\dfrac{f_{11}}{\eta}a\Delta_{13 } \pm2\Delta_{13} - a\Delta_{23}\bigg]
+ af_{31}\frac{\Delta_{12}}{\eta}\pm 2\Delta_{23}. 
\end{equation}
Substituting (\ref{Eq1f11neq0-2}) in (\ref{Eq2f11neq0-2}),  we obtain  
\[
F = \pm\dfrac{1}{f_{11,z_1}} \left(1+ \dfrac{f_{11}^2}{\eta^2}\right)
\left( af_{31}\frac{\Delta_{12}}{2\eta}\mp \Delta_{23}\right),
\]
 which simplifies to 
\begin{equation}\label{FCont}
F=\dfrac{(f_{11}^2 + \eta^2)f_{32}}{\eta f_{11,z_1}}.
\end{equation}

Observe that we are considering $f_{22}\neq 0$ and $f_{11,z_1}\neq 0$. 
A straightforward computation shows that (\ref{FCont}) leads to a contradiction 
for equation (\ref{eqF}) with $f_{ij}$ as in 
(\ref{fijF''}) with  $B\neq 0$ and equations (\ref{eqexp}), (\ref{eqLin}) with 
$f_{ij}$ given as in (\ref{fijexp}) - (\ref{fijLin3}). 
This concludes the proof of Lemma \ref{Claim I}. \end{proof}

\vspace{.2in}

\begin{Lem}\label{Claim II}
Consider an equation $u_{xt}=F(u,u_x)$ describing $\eta$ pseudo-spherical surfaces, with 1-forms 
$\omega^i$ as in (\ref{forms}) where the functions $f_{ij}$ are given by (\ref{fijF''})-(\ref{fijLin3}). Assume there is a local isometric immersion of 
a pseudo-spherical surface, determined by a solution $u(x,t)$, for which the coefficients $a,\, b, \,c$ of the second fundamental form depend on a jet of finite order of $u$.  
 If 
\begin{equation}\label{Eqcf11neq0}
c +  \bigg(\dfrac{f_{11}}{\eta}\bigg)^2a +  2\dfrac{f_{11}}{\eta} b \neq 0,  
\end{equation}
 holds then $a,\, b$ and $c$ are functions of $x$ and $t$, and thus universal.
 \end{Lem}
 
\vspace{.1in} 
 
 \begin{proof}
 If  (\ref{Eqcf11neq0}) holds 
then, it follows from  Lemma \ref{Lem_ac} that $c\neq 0$ and $ f_{11}\neq 0$.
Moreover, from  (\ref{Eqpro}) we get  
$a_{w_\ell} = 0$ 
and hence (\ref{Cons1}) implies that $b_{w_\ell} = c_{w_\ell}=0$.

If $\ell=0$, then $a, b$, and $c$ are functions of $x$ and $t$, and thus universal. 
 If $\ell\geqslant 1$, then consecutive differentiation of (\ref{Eq1k}), (\ref{Eq2k}) and (\ref{Gauss}) with respect to $w_\ell, \dots w_1$ lead to $a_{w_i} = b_{w_i} = c_{w_i} = 0$ for $i=0, \dots, \ell$.  In particular, 
$a$, $b$ and $c$ do not depend on $z_0$.  Therefore, $a, b$, and $c$ are functions of $x, t, z_1, \dots, z_\ell$. Differentiating (\ref{Eq1k}) and (\ref{Eq2k}) 
with respect to $z_{\ell+1}$ leads to 
\begin{equation}\label{Cons2}
f_{12}a_{z_\ell} + f_{22} b_{z_\ell} = 0  \quad  \text{ and }\quad f_{12}b_{z_\ell} + f_{22} c_{z_\ell} = 0.
\end{equation} 
Differentiation of the Gauss equation (\ref{Gauss}) with respect to $z_\ell$ gives 
\begin{equation}\label{Gaussdifk}
ca_{z_\ell} + ac_{z_\ell} - 2bb_{z_\ell} = 0.  
\end{equation}

\vspace{.1in}

 If $f_{22} = 0$, which is the case for  equation (\ref{eqLin}) with 
$f_{ij}$ as in (\ref{fijLin3}), since $f_{12}\neq 0$, (\ref{Cons2}) implies that 
$a_{z_\ell} = b_{z_\ell} = 0$, and (\ref{Gaussdifk}) leads to  
  $a c_{z_\ell} = 0$. From Lemma (\ref{Lem_ac}) we have $a\neq 0$, hence $c_{z_\ell}=0$. 
  Successive differentiation of (\ref{Eq1k}), (\ref{Eq2k}) and (\ref{Gauss}) with respect to $z_\ell, \dots , z_2$ leads to $a_{z_i} = b_{z_i} =0$, and hence $c_{z_i} = 0$ for $i=1, \dots , \ell$. Therefore, $a, b$, and $c$ are functions of $x$ and $t$. 

\vspace{.1in}

 If $f_{22} \neq  0$, then (\ref{Cons2}) leads to 
\begin{equation}\label{bzk_czk}
b_{z_\ell} = -\dfrac{f_{12}}{f_{22}}a_{z_\ell} \quad \text{ and } \quad c_{z_\ell} = \dfrac{f_{12}^2}{f_{22}^2} a_{z_\ell}.
\end{equation}
Substituting these expressions into (\ref{Gaussdifk}) we get  
\begin{equation}\label{Pro2}
\bigg[ c +  \bigg(\dfrac{f_{12}}{f_{22}}\bigg)^2 a +  2\dfrac{f_{12}}{f_{22}} b\bigg]a_{z_\ell} = 0.
\end{equation}

If 
\begin{equation}\label{cf12neq0}
 c +  \bigg(\dfrac{f_{12}}{f_{22}}\bigg)^2 a +  2\dfrac{f_{12}}{f_{22}} b \neq 0, 
\end{equation}
then $a_{z_\ell}=0$ and (\ref{bzk_czk}) implies that $b_{z_\ell}=c_{z_\ell}=0$. 
 Consecutive differentiations of (\ref{Eq1k}) and (\ref{Eq2k}) with respect to $z_\ell, \dots, z_2$ lead to $a_{z_i} = b_{z_i} = c_{z_i} = 0$ for $i=1, \dots, \ell$, and hence, $a, b$ and $c$ are functions of $x$ and $t$ only.

If 
\begin{equation}\label{cf120}
 c +  \bigg(\dfrac{f_{12}}{f_{22}}\bigg)^2 a +  2\dfrac{f_{12}}{f_{22}} b = 0 
\end{equation}
on a non empty open set, then Lemma \ref{Lem_ac} and (\ref{Eqcf11neq0})  imply that $c\neq 0$ and hence $f_{12}\neq 0$.
It follows from (\ref{cf120}) and (\ref{Gauss}) that 
\begin{equation}\label{b=c=Case2}
b= \pm 1 - \dfrac{f_{12}}{f_{22}}a \quad  \text{ and }  \quad 
c = \bigg(\dfrac{f_{12}}{f_{22}}\bigg)^2a  \mp 2\dfrac{f_{12}}{f_{22}}.
\end{equation} 
Therefore, 
\begin{eqnarray*}
f_{11}D_t a + \eta D_t b& =& \dfrac{\Delta_{12}}{f_{22}}D_t a - \eta a 
    \bigg(\dfrac{f_{12}}{f_{22}} \bigg)_{z_0} w_1, \\
f_{12}D_x a + f_{22}D_x b &=&  - a f_{22}\bigg(\dfrac{f_{12}}{f_{22}} \bigg)_{z_0}z_1, \\  
f_{11}D_t b + nD_t c &=&\left(\frac{2\eta f_{12}}{f_{22}}a-f_{11}a \mp 2\eta \right) 
    \bigg(\dfrac{f_{12}}{f_{22}} \bigg)_{z_0} w_1 -\frac{f_{12}}{f_{22}^2}\Delta_{12}D_ta,\\
f_{12}D_x b + f_{22}D_x c& =& (af_{12}\mp 2f_{22})\bigg(\dfrac{f_{12}}{f_{22}} \bigg)_{z_0}z_1. 
\end{eqnarray*}

Therefore, equation (\ref{EQ1}) becomes 
\begin{equation*}\label{EQ1-1}
\dfrac{\Delta_{12}}{f_{22}} D_t a -\eta a \bigg(\dfrac{f_{12}}{f_{22}} \bigg)_{z_0}w_1 + a f_{22}\bigg(\dfrac{f_{12}}{f_{22}} \bigg)_{z_0}z_1  - 2b \Delta_{13}  +(a-c) \Delta_{23} = 0
\end{equation*}
and (\ref{EQ2}) becomes
\begin{equation*}
\label{EQ2-1}
 -\dfrac{f_{12}}{f_{22}}\dfrac{\Delta_{12}}{f_{22}}D_t a + 
\left\{ \left(\frac{2\eta f_{12}}{f_{22}}a-f_{11}a \mp 2\eta \right) w_1  
-(af_{12}\mp 2f_{22})z_1\right\} \bigg(\dfrac{f_{12}}{f_{22}} \bigg)_{z_0} 
+ (a-c) \Delta_{13} + 2b\Delta_{23}=0.
\end{equation*}
Differentiating the first equation  with respect to $w_1$ leads to 
$\eta a
(\frac{f_{12}}{f_{22}})_{z_0}=0$. Since $\eta a\neq 0$ we have $ (\frac{f_{12}}{f_{22}})_{z_0}=0$ and the equations reduce to
\begin{eqnarray}\label{Cons1-2}
\dfrac{\Delta_{12}}{f_{22}} D_t a  - 2b \Delta_{13}  +(a-c) \Delta_{23} = 0,\\ 
\label{Cons2-2}
-\dfrac{f_{12}}{f_{22}}\dfrac{\Delta_{12}}{f_{22}}D_t a  + (a-c) \Delta_{13} + 2b\Delta_{23}=0, 
\end{eqnarray}
Adding (\ref{Cons1-2}) multiplied by $f_{12}/f_{22}$ with (\ref{Cons2-2}) we get  
\[ 
 a\Delta_{13}+\left( \pm 2 -\frac{f_{12}}{f_{22}}a \right)\Delta_{23}=0, 
\]
which reduces to 
\begin{equation}\label{a}
\frac{f_{32}}{f_{22}}\Delta_{12} a\pm 2\Delta_{23}=0.
\end{equation}
Observe that we have $f_{22}\neq 0$, $ f_{12}\neq 0$ and $(f_{12}/f_{22})_{z_0}=0$. Therefore, 
the only equation that satisfies these conditions is (\ref{eqexp}) with $f_{ij}$ as in Lemma \ref{Lemexp}.

If $\gamma=1$, it follows from (\ref{fijexp1}) that $f_{32}\neq 0$ and  (\ref{a}) implies that 
$a$ is constant hence, $D_ta=0$. Therefore, (\ref{Cons1-2}) and (\ref{Cons2-2}) reduce to 
\[
\left(
\begin{array}{cc}
-2b & a-c \\
a-c & 2b
\end{array}\right)
\left(
\begin{array}{c}
\Delta_{13} \\
\Delta_{23}
\end{array}\right)=
\left(
\begin{array}{c}
0 \\
0
\end{array}\right).
\]
 It follows from (\ref{deltaijneq}) that $b=0$ and $a=c$, which contradicts the Gauss equation.
   
If $\gamma\neq 1$, the functions $f_{ij}$ are given by (\ref{fijexp}). If $f_{32}=0$ then 
$B=0$ and (\ref{a}) implies that $\Delta_{23}=0$. Then it follows from the expression of $\Delta_{23}$ 
that $A=0$, which contradicts the fact that $A^2-B^2\neq 0$. 
If $f_{32}\neq 0$ i.e., $B\neq 0$, then  (\ref{a}) implies that 
\begin{equation}\label{aFinal}
a=\mp 2 \frac{\Delta_{23}f_{22}}{\Delta_{12}f_{32}}.
\end{equation}
Substituting the expressions of $b$ and $c$ as in (\ref{b=c=Case2}) into (\ref{Cons1-2}) we get 
\begin{equation} 
\label{aDta2}
\dfrac{\Delta_{12}}{f_{22}}D_t a +2\left(\frac{f_{12}}{f_{22}}a\mp 1 \right) \Delta_{13}+
     \left[a-\left(\frac{f_{12}}{f_{22}} \right)^2a\pm 2\frac{f_{12}}{f_{22}}\right]\Delta_{23}=0.
\end{equation}      
Computing the total derivative of $a$ with respect to $t$, using the expression of $a$ as in (\ref{aFinal}), equation (\ref{aDta2}) leads to 
\[
F (\Delta_{23,z_1}\Delta_{12} - \Delta_{12,z_1}\Delta_{23}) = -f_{22}(\Delta_{13}^2 + \Delta_{23}^2), 
\]
which  in view of (\ref{eqexp}) and (\ref{fijexp}) reduces to 
$ (B^2-A^2\gamma)z_1^2-A^2\beta=0$, which is also a contradiction. Therefore, we conclude that  
 the system (\ref{EQ1}),  (\ref{EQ2}) and the Gauss equation is an inconsistent system. 
 This concludes the proof of Lemma \ref{Claim II}.  \end{proof}

\begin{Lem}\label{CoeffUniversalEqTypei} Consider the equation $u_{xt}=F(u,u_x)$ which describes $\eta$ pseudo-spherical surfaces  where $F$ is given by 
(\ref{eqF}) and  $f_{ij}$ as in (\ref{fijF''}).  If the coefficients of the second fundamental form of the isometric immersion  in $\mathbb{R}^3$ of the pseudo-spherical surface,  determined  by a solution $u$,   are universal, then the system of equations (\ref{EQ1}), (\ref{EQ2}) and the Gauss equation (\ref{Gauss}) is inconsistent.   
\end{Lem}

\begin{proof} If the coefficients of the second fundamental form of the isometric immersion of the $\eta$ pseudo-spherical surfaces described by the differential  equation are universal, then  equations (\ref{EQ1}) and  (\ref{EQ2}) reduce to:
 \begin{eqnarray*}\label{EQ1-Univ}
f_{11}a_t + \eta b_t - f_{12}a_x - f_{22}b_x - 2b \Delta_{13} + (a-c)\Delta_{23} = 0\\\label{EQ2-Univ}
f_{11}b_t + \eta c_t - f_{12}b_x - f_{22}c_x +(a-c) \Delta_{13} + 2b\Delta_{23} = 0, 
\end{eqnarray*}
where $f_{ij}$ are given by (\ref{fijF''}). Differentiating 
both equations with respect to $z_1$ leads to 
\begin{eqnarray}\label{EQ1-iG-Univ}
-\alpha B a_t - 2b\dfrac{\alpha (QF' - \eta F)}{Q^2\alpha + \eta^2} + (a-c)\dfrac{\alpha A (\eta F' + \alpha Q F)}{Q^2 \alpha + \eta^2} = 0\\ \label{EQ2-iG-Univ}
-\alpha B b_t + (a-c)\dfrac{\alpha (QF' - \eta F)}{Q^2\alpha + \eta^2} + 2b\dfrac{\alpha A (\eta F' + \alpha Q F)}{Q^2 \alpha + \eta^2} = 0
\end{eqnarray}
Multiplying  (\ref{EQ1-iG-Univ}) and (\ref{EQ2-iG-Univ}) by  $Q^2\alpha + \eta^2 / \alpha$, and differentiating with respect to $z_0$, and taking into account that $F'' = -\alpha F$, we obtain
\[
\left(
\begin{array}{cc}
2b & \alpha A(a-c)\\ -(a-c) & 2\alpha A b
\end{array} \right)
\left(
\begin{array}{c}
\alpha QF+\eta F' \\ QF'-\eta F
\end{array} \right) =
\left(
\begin{array}{c}
0\\ 0
\end{array} \right).
\]
Since $\alpha QF+\eta F'$ and  $QF'-\eta F$ are not zero, we conclude that 
$\alpha A[4b^2+(a-c)^2]=0$. If $b=0$ and $a=c$ then Gauss equation leads to a contradiction.
If $A=0$ then  equations (\ref{EQ1-iG-Univ}) and (\ref{EQ2-iG-Univ}) reduce to 
\begin{eqnarray*}
-\alpha B a_t - 2b\dfrac{\alpha (QF' - \eta F)}{Q^2\alpha + \eta^2}  = 0,\\ 
-\alpha B b_t + (a-c)\dfrac{\alpha (QF' - \eta F)}{Q^2\alpha + \eta^2} = 0.
\end{eqnarray*}
taking derivative with respect to $z_0$ of both equations, we conclude that $b=a-c=0$ which is again a contradiction. Therefore, the system (\ref{EQ1}), (\ref{EQ2}) and the Gauss equation is inconsistent.  \end{proof}


\begin{Prop}\label{Propi}Consider an equation 
\begin{equation*}
u_{xt} = F(u),  \mbox{ with } \quad F''+\alpha F=0, \quad \alpha\neq 0, 
\end{equation*}
describing $\eta$ pseudo-spherical surfaces
with $f_{ij}$ given by 
(\ref{fijF''}). 
There exists a local isometric immersion in $\mathbb{R}^3$ of a pseudo-spherical surface, defined by a solution $u$, 
for which the coefficients of the second fundamental form  depend on a jet of finite order of $u$, that is, $a, b$ and $c$ depend on $x, t, u, w_1\dots, \partial^\ell u/\partial x^\ell, w_\ell$, where $\ell$ is finite if, and only if,  $\alpha>0$ and  $f_{ij}$ are given by (\ref{fijF''QA}),  
 $a, b, c$ depend on the jet of order zero of $u$ and  are given by  (\ref{newabc}). 
  \end{Prop}

\begin{proof}
Assume the local isometric immersion exists. If 
$c+(f_{11}/\eta)^2a+2f_{11}b/\eta=0$ on a non empty open set, then it follows from Lemma \ref{Claim I} that  
$B=0$, i.e. $\alpha>0$ and  $f_{ij}$ are given by  (\ref{fijF''QA}). 
Moreover, $a, b, c$ depend on the jet of order zero of $u$ and  are given by (\ref{newabc}).  
 If 
$c+(f_{11}/\eta)^2a+2f_{11}b/\eta\neq 0$, then Lemma \ref{Claim II} implies that $ a, b, c $ are 
universal. However, it follows from Lemma \ref{CoeffUniversalEqTypei} that such an immersion does not exist.

Conversely, a straightforward computation shows that if $f_{ij}$ are given as in  (\ref{fijF''QA}) and $a,b, c$ as in  (\ref{newabc}), then 
the connection forms $\omega_1^3$ and $\omega_2^3$ given by (\ref{w13_w23}) satisfy the structure 
equations (\ref{Codazzi}) of an immersion in $\mathbb{R}^3$ and the Gauss equation (\ref{Gauss}).  \end{proof}

\begin{Prop}\label{Propii} Consider an equation of type  
$u_{xt}= \nu e^{\delta u}\sqrt{\beta +\gamma u_x^2}$  describing $\eta$ pseudo-spherical surfaces, 
with $f_{ij}$ given by Lemma \ref{Lemexp}. There is no local isometric immersion in $\mathbb{R}^3$ of a 
pseudo-spherical surface determined by a solution $u$ of the equation, for which  the coefficients of the second fundamental form depend on a jet of finite order of $u$.  
  \end{Prop}

\begin{proof} If the immersion exists, then Lemma \ref{Claim I} ii) implies that 
$c+(f_{11}/\eta)^2a+2f_{11}b/\eta\neq 0$, and  it follows from Lemma \ref{Claim II}  that $ a, b, c $ are 
universal.  Therefore,  equations (\ref{EQ1}) and  (\ref{EQ2}) reduce to
\begin{eqnarray*}
f_{11}a_t + \eta b_t - f_{12}a_x - f_{22}b_x - 2b \Delta_{13} + (a-c)\Delta_{23} = 0, \\ 
f_{11}b_t + \eta c_t - f_{12}b_x - f_{22}c_x +(a-c) \Delta_{13} + 2b\Delta_{23} = 0, 
\end{eqnarray*}
where $f_{ij}$ are given by (\ref{fijexp}) if $\gamma\neq 1$  and (\ref{fijexp1}) if $\gamma=1$.
Differentiating these equations with respect to $z_1$ and then with respect to $z_0$ leads to 
\[\left(\begin{array}{cc}
-2b & a-c\\
a-c & 2b \end{array}\right)
\left(\begin{array}{c}
\Delta_{13,z_1z_0} \\ \Delta_{23,z_1z_0}
\end{array}\right)\,=\,
\left(\begin{array}{c}
0 \\ 0
\end{array}\right).
\]
In both cases, i.e., $\gamma=1$ or $\gamma\neq 1$, since $\Delta_{13,z_1z_0} \Delta_{23,z_1z_0}\neq 0$,  these equations imply that
  $b=0$ and $a=c$ which is inconsistent with the Gauss equation.  \end{proof}

\begin{Prop}\label{Propiii}  Consider an equation $u_{xt}=\lambda u+\xi u_x+\tau $ describing  $\eta$-pseudospherical surfaces with $f_{ij}$ given by (\ref{fijLin1})-(\ref{fijLin3}). There exists a local 
isometric immersion  in $\mathbb{R}^3$ of a pseudo-spherical surface, defined by a solution $u$,  for which    the coefficients of the second fundamental form $a$, $b$, $c$ 
  depend of a jet of finite order of $u$ if, and only if, $\lambda$, $\xi$ and $\tau$ do not vanish simultaneously and $a$, $b$, $c$ 
 are universal and given by:
\begin{enumerate}
\item [i)] When  $\lambda \neq 0$, 
\begin{equation} \label{abcLin2}
a= \sqrt{l L(x,t) - \gamma^2 L^2(x,t) - 1},\qquad 
b = \gamma L(x,t), 
\qquad c=\frac{b^2-1}{a}, 
\end{equation}
where $L(x,t)= e^{\pm 2 [\eta x + (\lambda / \eta \mp \zeta) t ]}$ $l, \gamma \in \mathbb{R}$ and $l^2 > 4\gamma^2$ and the  $1$-forms are defined on a strip of $\mathbb{R}$ where
\begin{equation}\label{stripLin2}
\log \sqrt{ \dfrac{l - \sqrt{l^2 - 4 \gamma^2}}{2\gamma^2 }} < \pm [\eta x + (\lambda / \eta \mp \zeta) t ] <  \log \sqrt{ \dfrac{l + \sqrt{l^2 - 4 \gamma^2}}{2\gamma^2 }}. 
\end{equation}
 
\item [ii)] When  $\lambda = 0$ and $\xi^2+\tau^2\neq 0$,  
\begin{equation}\label{abcLin3}
a= \sqrt{l e^{2 \eta x } - \gamma^2 e^{4  \eta x } - 1},\qquad 
b = \gamma e^{ 2  \eta x  },\qquad c=\frac{b^2-1}{a}, 
\end{equation}

$l, \gamma \in \mathbb{R}$ and $l^2 > 4\gamma^2$ and the  $1$-forms are defined on a strip of $\mathbb{R}^2$ where
\begin{equation}\label{stripLin3}
\log \sqrt{ \dfrac{l - \sqrt{l^2 - 4 \gamma^2}}{2\gamma^2 }} <  \eta x <  \log \sqrt{ \dfrac{l + \sqrt{l^2 - 4 \gamma^2}}{2\gamma^2 }}.  
\end{equation}

\end{enumerate}
 Moreover, the constants $l$ and $\gamma$ have to be chosen so that the strip intersects the domain of the solution of the evolution equation. 
\end{Prop}

\begin{proof}If the coefficients of the second fundamental form of the local isometric immersion of $\eta$ pseudo-spherical surfaces described by the equation of type iii) depend of a jet of finite order of $u$,  then they are universal by Lemmas \ref{Claim I} and \ref{Claim II}, and hence (\ref{EQ1}) and (\ref{EQ2}) becomes 
\begin{eqnarray}\label{EQ1iii-Univ}
f_{11}a_t + \eta b_t - f_{12}a_x - f_{22}b_x- 2b \Delta_{13}+(a-c)\Delta_{23} = 0,\\\label{EQ2iii-Univ} 
f_{11}b_t + \eta c_t - f_{12}b_x - f_{22}c_x + (a-c) \Delta_{13}+ 2b\Delta_{23} = 0.
\end{eqnarray}

\vspace{.1in}
 
 If $\lambda=\xi=\tau=0$ and $f_{ij}$ are given by (\ref{fijLin1}) then taking the derivative of both equations with respect to $z_0$, and using the fact that $\Delta_{13}=e^{z_0}z_1$ and $\Delta_{23}=0$ 
 we get 
 \begin{eqnarray*}
 b_x+2bz_1=0,\\
 c_x-(a-c)z_1=0.
 \end{eqnarray*} 
Since $a,b,c $ are universal we conclude that $b=0$ and $a=c$ which contradicts Gauss equation.
Therefore the immersion does not exit. 

\vspace{.1in}
  
i) If $\lambda \neq 0$ and  the functions $f_{ij}$ are as in (\ref{fijLin2}) then $\Delta_{13}=0$. 
 Differentiating (\ref{EQ1iii-Univ}) and (\ref{EQ2iii-Univ}) with respect to $z_1$ leads to (after dividing by $f_{11, z_1}$)
\begin{eqnarray}\label{iiiG-at}
a_t  &=& \pm f_{22}(a-c),\\\label{iiiG-bt}
b_t &=& \pm 2 b f_{22}.
\end{eqnarray}
Differentiating (\ref{EQ1iii-Univ}) and (\ref{EQ2iii-Univ}) 
with respect to $z_0$ leads to (after dividing by $f_{12,z_0}$)
\begin{eqnarray}\label{iiiG-ax}
a_x &=& \pm \eta (a-c),\\
b_x & =& \pm 2 \eta b. \label{iiiG-bx}
\end{eqnarray}
and hence, (\ref{EQ1iii-Univ}) and (\ref{EQ2iii-Univ}) reduce to 
\begin{eqnarray}\label{iiiG-btbx}
\eta b_t - f_{22}b_x = 0,\\\label{iiiG-ctcx}
\eta c_t - f_{22}c_x = 0. 
\end{eqnarray}
The equations (\ref{iiiG-at}), (\ref{iiiG-bt}), (\ref{iiiG-ax}), (\ref{iiiG-bx}), (\ref{iiiG-btbx}), and (\ref{iiiG-ctcx}) are the same as (\ref{Cat=0}), (\ref{Cbt=0}), (\ref{Cax=0}), (\ref {Cbx=0}), (\ref{btbx=0}), and  (\ref{ctcx=0}) respectively, since $f_{22}$ is contant. Therefore, $a$ is as in  (\ref{auniversal}), $b$ is as in (\ref{buniversal}), and $c$ is as in (\ref{cuniversal}) and are subject to (\ref{strip}),  where $\lambda$ is replaced by $f_{22} = \lambda/\eta \mp \zeta$. Therefore, we obtain $a,b,c$ given as in (\ref{abcLin2}) 
defined on the strip (\ref{stripLin2}). 

\vspace{.1in}

ii) If $\lambda = 0$, $\xi^2+\tau^2\neq 0$ and the functions $f_{ij}$ are as in (\ref{fijLin3}), then $\Delta_{13}=0$ and 
$\Delta_{23}=1$, hence   
  (\ref{EQ1iii-Univ}) and (\ref{EQ2iii-Univ}) reduce to 
\begin{eqnarray}\label{EQ1iiiNG-Univ-1}
f_{11}a_t + \eta b_t - f_{12}a_x + (a-c) = 0,\\\label{EQ2iiiNG-Univ-1}
f_{11}b_t + \eta c_t - f_{12}b_x + 2b= 0. 
\end{eqnarray}
Differentiating  with respect to $z_1$ leads to $a_t = b_t = 0$. Since from Lemma \ref{Lem_ac} we have $a\neq 0$, Gauss equation implies that $c_t=0$ 
 and thus (\ref{EQ1iiiNG-Univ-1})  and (\ref{EQ2iiiNG-Univ-1})  become
\begin{eqnarray*}
a_x & = & \eta (a-c),\\
b_x & = & 2\eta b,
\end{eqnarray*}
where $c=(b^2-1)/a$. The arguments used in the proof of Proposition \ref{propabcexpluniversal}, with $\lambda=0$ and $\pm$ replaced by $+$, imply that $a, b, c$ are given by (\ref{abcLin3}), that are defined on the strip 
given by (\ref{stripLin3}).

The converse follows from a straightforward computation.   \end{proof}

Finally, the proof of Theorem \ref{HyperbRes} follows from Propositions \ref{Propi}, \ref{Propii} and \ref{Propiii}.   \hfill $\Box$


\vspace{.35in}

\noindent 
Nabil Kahouadji, Department of Mathematics, Northwestern University, USA\\
e-mail: nabil@math.northwestern.edu
\\ \\
Niky Kamran, Department of Mathematics and Statistics, McGill University, Canada\\ 
e-mail: nkamran@math.mcgill.ca
\\ \\ 
Keti Tenenblat, Department of Mathematics, Universidade de Bras\'\i lia, Brazil\\ 
e-mail: K.Tenenblat@mat.unb.br

\end{document}